\documentclass[11pt]{article}
 
\usepackage[compress]{natbib}
\usepackage{amsmath,amssymb,amsthm}
\usepackage{color} 
\usepackage[margin=1in]{geometry}
\usepackage{setspace}
\usepackage{amsthm, amsmath, amssymb, amsfonts}
\usepackage{bm}
\usepackage[dvipsnames,table,xcdraw]{xcolor}
\usepackage{lipsum}
\usepackage{appendix}
\usepackage{amsfonts}
\usepackage{graphicx}
\usepackage{epstopdf}
\usepackage{algorithm}
\usepackage{algorithmic}
\usepackage{enumitem}   
\ifpdf
  \DeclareGraphicsExtensions{.eps,.pdf,.png,.jpg}
\else
  \DeclareGraphicsExtensions{.eps}
\fi

\allowdisplaybreaks[4]

\usepackage{amsopn}
\usepackage{mathtools}   
\usepackage{url}            
\usepackage{booktabs}       

\usepackage{subfig}

\usepackage{mathrsfs}
\usepackage{hyperref} 

\usepackage{multirow}

\usepackage{tabularx}
\usepackage{colortbl}
\usepackage{textcomp}

\usepackage{caption}
\usepackage{subcaption}
 
\usepackage{tcolorbox} 
\usepackage{xr}

\usepackage{caption}
\usepackage{url}            
\usepackage{booktabs}       
\usepackage{nicefrac}       
\usepackage{graphicx}
\usepackage{caption}
\usepackage{thm-restate}
\usepackage{bbm}
\usepackage[export]{adjustbox}

\usepackage{microtype}
\usepackage{mathpazo} 
\usepackage[scaled]{helvet} 
\usepackage{courier} 
\normalfont
\usepackage[T1]{fontenc}

\usepackage{tikz}
\usetikzlibrary{shadows}
\usetikzlibrary{arrows,positioning} 
\def\arraystretch{1.5}
\usepackage{tcolorbox}

\def\O#1{\text{\ding{\the\numexpr#1+171}}}

\definecolor{amethyst}{rgb}{0.6, 0.4, 0.8}
\newcount\Comments
\newcommand{\kibitz}[2]{\ifnum\Comments=1{\textcolor{#1}{\textsf{\footnotesize #2}}}\fi}

\usepackage{graphicx} 
\usepackage{latexsym}
\usepackage{mathrsfs}
\usepackage{mathtools}
\usepackage{appendix}
\usepackage{multirow}
\usepackage{makecell}
\usepackage{mathtools}
\usepackage{amssymb}
\usepackage{bm} 
\usepackage{pifont} 
 
\usepackage{amsmath}   

\usepackage{nameref}
\hypersetup{
   linkcolor={red!80!black},citecolor={blue!70!black}, urlcolor={blue!70!black},colorlinks, breaklinks
           }
\usepackage{cleveref}

\def\s{\bm{s}}

\def\x{\bm{x }}
\def\y{\bm{y}}
\def\z{\bm{z}}
\def\v{\bm{v}}
\def\u{\bm{u}}

\def\meps{\bm{\epsilon}}

\def\b1{\bm{1}}

\newcommand{\D}{\mathcal D}

\newcommand{\R}{\mathbb R}
\newcommand{\X}{\mathcal X}
\newcommand{\Y}{\mathcal Y}

\newcommand{\mO}{\mathcal O}

\newcommand{\U}{\mathcal U}

\newcommand{\bz}{\mathbf 0}

\DeclareMathOperator{\prox}{prox}
\DeclareMathOperator{\proj}{proj}

\DeclareMathOperator{\diam}{diam}
\DeclareMathOperator{\dist}{dist}
 
\DeclareMathOperator*{\argmin}{argmin}
\DeclareMathOperator*{\argmax}{argmax}

\usepackage{thmtools}
 
\declaretheoremstyle[parent=section]{definitionwithend}
\declaretheorem[style=definitionwithend]{corollary}
\declaretheorem[style=definitionwithend]{theorem}
\declaretheorem[style=definitionwithend]{proposition}
\declaretheorem[style=definitionwithend]{definition}
\declaretheorem[style=definitionwithend]{assumption}

\declaretheorem[style=definitionwithend]{remark}

\declaretheorem[style=definitionwithend]{lemma}
 
\declaretheorem[style=definitionwithend]{condition}

\usepackage{appendix}
\usepackage{pifont}
\title{ 
 Doubly Smoothed Optimistic Gradients: A Universal Approach for Smooth Minimax Problems 
}

\date{\today}

\author{%
    Taoli Zheng\thanks{Department of Systems Engineering and Engineering Management, The Chinese University of Hong Kong, Shatin, NT, Hong Kong.  \texttt{tlzheng@se.cuhk.edu.hk}} \and
    Anthony Man-Cho So\thanks{Department of Systems Engineering and Engineering Management, The Chinese University of Hong Kong, Shatin, NT, Hong Kong. \texttt{manchoso@se.cuhk.edu.hk}}\and
    Jiajin Li\thanks{Sauder School of Business, University of British Columbia, Vancouver, BC, Canada. \texttt{jiajin.li@sauder.ubc.ca}} 
}
\begin{document}
\maketitle
 
\begin{abstract}
Smooth minimax optimization problems play a central role in a wide range of applications, including machine learning, game theory, and operations research. 
However, existing algorithmic frameworks  vary significantly depending on the problem structure --- whether it is convex-concave, nonconvex-concave, convex-nonconcave, or even nonconvex-nonconcave with additional regularity conditions. In particular, this diversity complicates the tuning of step-sizes since even verifying convexity (or concavity) assumptions is challenging and problem-dependent.
We introduce a universal and single-loop  algorithm, Doubly Smoothed Optimistic Gradient Descent Ascent (DS-OGDA), that applies to a broad class of smooth minimax problems. Specifically, this class includes  convex-concave, nonconvex-concave, convex-nonconcave, and nonconvex-nonconcave minimax optimization problems satisfying a one-sided Kurdyka-\L{}ojasiewicz (K\L{}) property.  
DS-OGDA works with a universal single set of parameters for all problems in this class, eliminating the need for prior structural knowledge to determine step-sizes. 
Moreover, when a particular problem structure in our class is specified, DS-OGDA achieves optimal or best-known performance guarantees. 
Overall, our results provide a comprehensive and versatile framework for smooth minimax optimization, bridging the gap between convex and nonconvex problem structures and simplifying the choice of algorithmic strategies across diverse applications.
\end{abstract}
 
\begin{keywords}
Smooth Minimax Optimization, First-Order Methods   
\end{keywords}

\section{Introduction}\label{sec:intro}
In this paper, we study the smooth minimax problem:
\begin{equation}\label{eq:prob}
\min_{\x\in \X}\max_{\y\in \Y} f(\x,\y), \tag{P}
\end{equation}
where $f:\R^n\times \R^d\rightarrow\R$ is continuously differentiable, and $\X\subseteq \R^n, \Y \subseteq \R^d$ are nonempty compact convex sets.  
Problem \eqref{eq:prob} has gained increasing attention due to its extensive applications in machine learning and operations research, including generative adversarial networks \citep{arjovsky2017wasserstein,goodfellow2020generative}, multi-agent reinforcement learning \citep{omidshafiei2017deep,dai2018sbeed}, and (distributionally) robust optimization \citep{ben2009robust,bertsimas2011theory,kuhn2019wasserstein,rahimian2019distributionally,blanchet2024distributionally}.
 
Within this growing body of work, researchers have design tailored  algorithmic frameworks to solve problems under various structural assumptions. These assumptions broadly fall into two categories: (i) variational inequality (VI)-based conditions~\citep{diakonikolas2021efficient,cai2022accelerated,pethick2022escaping,bohm2023solving,cai2024accelerated}, and (ii) one-sided dominance conditions~\citep{yang2022faster,li2022nonsmooth,zheng2023universal}. These dominance conditions include convexity (resp. concavity) and Kurdyka-\L{}ojasiewicz (K\L{}) properties imposed on the primal (resp. dual) side of the problem.
These two categories of assumptions lead to fundamentally different principles for the design of algorithmic frameworks. 
One-sided dominance conditions leverage the asymmetry between primal and dual updates, while VI-based frameworks treat the primal and dual variables symmetrically. For instance, the inherent symmetry of convex-concave problems are often handled using VI-based methods, whereas nonconvex-concave and convex-nonconcave problems rely on carefully balancing the primal and dual steps.
Although these algorithmic frameworks ensure convergence under their respective structural assumptions, a major practical challenge remains: Verifying whether a given problem satisfies a particular structural assumption is often difficult, if not infeasible. Without this knowledge, selecting the appropriate algorithm becomes a daunting task.

To address this, we aim to develop a general-purpose algorithm for smooth minimax problems that requires minimal problem-specific tuning. To that end, we define a structured function class $\mathcal{F}$ to contain all continuously differentiable functions $f:\R^n\times \R^d \rightarrow \R$ whose gradients are Lipschitz continuous with a known constant $L>0$. 
Within $\mathcal{F}$, we focus on five canonical subclasses: convex–concave (C–C), nonconvex–concave (NC–C), convex–nonconcave (C–NC), nonconvex–KŁ (NC–KŁ), and KŁ–nonconcave (KŁ–NC). The last two classes are defined by relaxing convexity or concavity assumptions and replacing them with a one-sided KŁ property in either the $\x$ or $\y$ variable, respectively. Given such a problem instance $(f, \mathcal{X}, \mathcal{Y})$ with $f \in \mathcal{F}$, our goal is to design an algorithm that returns an approximate stationary point of the minimax problem~\eqref{eq:prob}.

To that end, we identify two desirable properties that such an algorithm should satisfy.  
(i)~\textbf{Universal applicability}: A single step-size schedule should guarantee convergence across all problem types in $f\in\mathcal{F}$, without requiring the user to verify convexity, concavity, or KŁ properties before applying the algorithm.  
(ii)~\textbf{Optimal achievability}: When structural knowledge is available—e.g., the problem is known to be C-C—the algorithm should be able to adjust step-sizes accordingly to match the optimal or best-known  iteration complexities for that subclass. This ensures the algorithm performs optimally when possible, while remaining broadly applicable when structure is uncertain. Our dual criteria—universal applicability and optimal achievability—are conceptually with the “\emph{universal method}” developed for minimization problems \citep{nesterov2015universal,yurtsever2015universal}. These prior works focus on convex and Hölder smooth (possibly nonsmooth) functions in pure minimization settings and develop adaptive algorithms that automatically adjust their step-sizes to learn the unknown Hölder exponent and smoothness constants \citep{lan2015bundle,levy2018online, kavis2019unixgrad, ghadimi2019generalized,li2023simple}. In our criteria, this corresponds to simultaneously achieving universal applicability and optimal achievability without requiring parameter tuning.  However, the minimization setting is significantly more benign. In particular, there is little algorithmic or analytical distinction between convex and nonconvex smooth minimization: For instance, accelerated gradient methods satisfy both criteria in smooth minimization, applying uniformly and achieving optimal complexity when the smoothness constant is known.  In contrast, minimax problems exhibit a fundamental algorithmic divide between convex and nonconvex structures, with no unified framework currently available. 
Our work takes a first step toward bridging this gap. We propose a first-order algorithm that applies uniformly across all subclasses in  $\mathcal{F}$, operates under a fixed step-size policy, guarantees convergence to an approximate stationary point, and achieves optimal iteration complexity when structural information is available.

Defining an algorithm that achieves both (i) and (ii) proceeds in two steps. We first start by describing an algorithmic framework that achieves (i) on all subclasses of problems. This algorithmic framework is inspired by that defined in \citet{zheng2023universal} but allows for further adjustment to attain (ii). 

The first step involves identifying two key algorithmic features that ensure universal applicability across NC-C and C-NC problems. First, the algorithm should ideally be a \textit{single-loop method}.  For NC-C problems, methods often treat the inner maximization function  $\phi(\x) = \max_{\y \in \Y} f(\x, \y)$ as a nonsmooth minimization problem in $\x$, leading to nested-loop algorithms that are well-understood in theory~\citep{thekumparampil2019efficient,nouiehed2019solving,lin2020near,jin2020local,ostrovskii2021efficient}.  However, extending this approach to C-NC problems is challenging due to the nonconcavity of the inner maximization. 
For the second feature, it is highly advantageous for the algorithm to be \textit{symmetric}—that is, to update both primal and dual variables using the same schemes and step-sizes. This rationale comes from the observation that $\min_{\x \in \X} \max_{\y \in \Y} f(\x,\y) = -\max_{\x \in \X} \min_{\y \in \Y} -f(\x,\y)$, and that the game stationary points~\citep[Definition 7.1]{li2022nonsmooth} of the min-max and max-min problems are generally the same. 
A symmetric algorithm thus enables solving a C-NC problem using the same method as for its NC-C counterpart.

The Doubly Smoothed Gradient Descent Ascent (DS-GDA) method proposed in \citet{zheng2023universal} is, to our knowledge, the only algorithm satisfying both the single-loop and symmetric features for nonconvex minimax optimization problems.  Its key contribution is the introduction of a double smoothing mechanism,  which enables symmetric updates and thereby ensures universal applicability across all problem classes considered in this work. Beyond universal applicability, a next question is whether DS-GDA also achieves optimal iteration complexity for different problem classes. 
For NC-C, C-NC and NC-NC problems satisfying a one-sided K\L{} condition, 
DS-GDA is able to achieve the best-known complexity for single-loop methods \citep{li2022nonsmooth,xu2023unified,zhang2020single}. However, for C-C problems, it only achieves a suboptimal 
iteration complexity  $\mathcal{O}(\epsilon^{-2})$  for reaching an
$\epsilon$-saddle point~\citep[Definition 4]{lin2020near}, and we prove that this complexity is already tight. 
The suboptimality arises because DS-GDA essentially reduces to a vanilla GDA update on a regularized C-C problem, which is proven to yield an inferior convergence rate.
Achieving the optimal complexity $\mathcal{O}(\epsilon^{-1})$  for C-C problems \citep{zhang2022lower} requires an extrapolation step, as used in Extragradient (EG) \citep{korpelevich1976extragradient,mokhtari2020convergence} or Optimistic Gradient Descent Ascent (OGDA) \citep{mertikopoulos2019optimistic,nemirovski2004prox,daskalakis2017training,rakhlin2013online} methods. 
This motivates integrating the double smoothing idea from DS-GDA into an extrapolation-based framework to ensure both universal applicability and optimal achievability.  

We thus propose Doubly Smoothed Optimistic Gradient Descent Ascent (DS-OGDA), an extension of DS-GDA that incorporates an additional extrapolation step in both the primal and dual updates. While this modification may seem natural, it introduces significant technical challenges in the analysis.
Unlike the existing literature, which typically treats nonconvex and C-C problems separately, DS-OGDA must handle both scenarios simultaneously. The extrapolation introduces extra error terms that require a carefully constructed Lyapunov analysis for nonconvex cases. Meanwhile, the smoothing steps dramatically complicate the boundedness analysis of the iterates, necessitating more intricate proofs than those required for standard OGDA to address the C-C case. 

Overall, this paper shows that DS-OGDA, using a single set of parameters, achieves:
\begin{enumerate}[label=(\roman*)] 
\item an $\mathcal{O}(\epsilon^{-4})$ iteration complexity for NC-C and C-NC problems, as well as NC-NC problems satisfying a one-sided K\L{} condition (see Assumption \ref{ass:dual}(ii)). 
\item an $\mathcal{O}(\epsilon^{-2})$
iteration complexity for C-C problems.
\end{enumerate}

Moreover, when additional problem structure is available—such as knowledge of the K\L{} exponent or verification that the problem is C-C—the parameters of DS-OGDA can be accordingly adjusted to yield sharper, and in some cases, optimal iteration complexities. Under the one-sided K\L{} condition with exponent $\theta \in (0,1)$, DS-OGDA achieves an iteration complexity of $\mathcal{O}(\epsilon^{-(4\theta - 2)_+ - 2})$. In particular, for C-C problems, it attains the optimal rate of $\mathcal{O}(\epsilon^{-1})$. These results match the best-known or optimal iteration complexities in the literature, as summarized in Table~\ref{table:main_results}.

\begin{table}[ht]  
\renewcommand\arraystretch{0.4}
\resizebox{\columnwidth}{!}{
\begin{tabular}{c|c|c|c}
\hline
\cellcolor[HTML]{FFFFFF}\textbf{Settings} & \textbf{References}              & 
\textbf{ Game Sta.} &
\textbf{Loops}
\\

\hline
&  Prox-Method \citep{nemirovski2004prox}, EG/OGDA \citep{mokhtari2020convergence}   & \cellcolor[HTML]{FFFFFF}$\mO(\epsilon^{-1})^\star$   
& 1    
   \\ \cline{2-4} 
\multirow{-2}{*}{C-C}          & \cellcolor[HTML]{C0C0C0} \textbf{Theorem \ref{theorem:cc}} & \cellcolor[HTML]{C0C0C0}$\mO(\epsilon^{-1})^\star$      
& \cellcolor[HTML]{C0C0C0}1                                          \\ 
\hline
 &  DS-GDA \citep{zheng2023universal}  & \cellcolor[HTML]{FFFFFF}$\mO(\epsilon^{- (4\theta-2)_+-2})$       
 & 1                                                                   \\ \cline{2-4} 
\multirow{-2}{*}{NC-K\L{}}          & \cellcolor[HTML]{C0C0C0}  \textbf{Theorem \ref{theorem:2} (i) \& (ii)} & \cellcolor[HTML]{C0C0C0}$\mO(\epsilon^{- (4\theta-2)_+-2})$ 
& \cellcolor[HTML]{C0C0C0}1                                           \\ \hline
& Minimax-PPA \citep{lin2020near}, FNE Search \citep{ostrovskii2021efficient}   & $\tilde{\mO}(\epsilon^{-2.5})$ 
& 3                                                                   \\ \cline{2-4} 
    & Multi-Step Projected Gradient  \citep{nouiehed2019solving}   & $\tilde{\mO}(\epsilon^{-3.5})$  
    & 3          \\ \cline{2-4} 
    & AGP \citep{xu2023unified}   & $\mO(\epsilon^{-4})$     
    & 1                                                                   \\ \cline{2-4} 
    &Smoothed-GDA \citep{zhang2020single}  & $\mO(\epsilon^{-4})$    
    & 1                                                                   \\ \cline{2-4}   
     &EG/OGDA \citep{mahdavinia2022tight}  & $\mO(\epsilon^{-6})$    
    & 1 
    \\
    \cline{2-4}
\multirow{-8}{*}{NC-C}      
& \cellcolor[HTML]{C0C0C0} \textbf{Theorem \ref{theorem:2} (iii)} & \cellcolor[HTML]{C0C0C0}$\mO(\epsilon^{-4})$ 
& \cellcolor[HTML]{C0C0C0}1                                           \\ \hline
   & DS-GDA \citep{zheng2023universal}   & \cellcolor[HTML]{FFFFFF}$\mO(\epsilon^{- (4\theta-2)_+-2})$    
   & 1    
   \\ \cline{2-4} 
   \multirow{-2}{*}{K\L{}-NC}       & 
\cellcolor[HTML]{C0C0C0} 
\textbf{Corollary \ref{col:1} (i) \& (ii)} & \cellcolor[HTML]{C0C0C0}$\mO(\epsilon^{- (4\theta-2)_+-2})$
& \cellcolor[HTML]{C0C0C0}1                                           \\ \hline
 &  AGP \citep{xu2023unified}, DS-GDA \citep{zheng2023universal}    & $\mO(\epsilon^{-4})$      
 & 1                                                                   \\ \cline{2-4} 
\multirow{-2}{*}{C-NC}                                             & \cellcolor[HTML]{C0C0C0} \textbf{Corollary \ref{col:1} (iii)} & \cellcolor[HTML]{C0C0C0}$\mO(\epsilon^{-4})$       
& \cellcolor[HTML]{C0C0C0}1                                           \\ \hline
    
\cellcolor{blue!15}{\textbf{Universality} }    & \cellcolor{blue!15}{ \textbf{Theorem \ref{theorem:universal}}} & \cellcolor{blue!15}{${\mathcal{O}(\epsilon^{-4})}$/${\mathcal{O}(\epsilon^{-2})^\star}$} 
& \cellcolor{blue!15}{1}\\
\hline 
\end{tabular} 
}
\caption{
Comparison of iteration complexities for various state-of-the-art methods under different structural assumptions. For C-C problems, we consider convergence to saddle points, as indicated by $\star$. 
}
\label{table:main_results}
\vspace{-1.0em}
\end{table}

\noindent\textbf{Notation.}  
We use bold lowercase letters (e.g., $\x,\y$) to denote vectors, and calligraphic uppercase letters (e.g., $\X,\Y$) to denote sets. For a closed and convex set $\X \subseteq \mathbb{R}^n$, the indicator function $\iota_{\X}:\mathbb{R}^n \to \{0,+\infty\}$ is defined by $\iota_{\X}(\x) = 0$ if $\x \in \X$ and $+\infty$ otherwise. Its sub-differential at a point $\x \in \X$, denoted by $\partial \iota_{\X}(\x)$, is the normal cone to $\X$ at $\x$. For any $\x \in \mathbb{R}^n$, the distance to the set $\X$ is defined as $\dist(\x,\X)=\min_{\y\in \X}\|\x-\y\|$, where $\|\cdot\|$ denotes the Euclidean norm. We write $\diam(\X)=\sup_{\x,\y \in \X}\|\x-\y\|$ for the diameter of $\X$, and $\proj_{\X}(\z)$ for the projection of a point $\z$ onto $\X$. Given $r>0$, the proximal operator of a function (if well-defined) $f:\mathbb{R}^n \to \mathbb{R}$ at a point $\z$ is defined as $\prox_{\tfrac{1}{r}\cdot f}(\z)=\argmin_{\x\in \X} f(\x)+\tfrac{r}{2}\|\x-\z\|^2$. For a differentiable function $f:\mathbb{R}^n \times \mathbb{R}^d \to \mathbb{R}$, we use $\nabla_{\x} f(\x,\y)$ and $\nabla_{\y} f(\x,\y)$ to denote its partial gradients with respect to $\x$ and $\y$, respectively. The notation $a_{+}$ denotes the positive part $\max\{a,0\}$. Finally, the Cartesian product of two sets $\X$ and $\Y$ is denoted by $\X\times\Y$.

\vspace{2mm}
\noindent \textbf{Organization.} 
The remainder of this paper is organized as follows. Section \ref{sec:main} introduces the proposed DS-OGDA method and presents our main theoretical results. Section \ref{sec:proof} provides detailed proofs. In Section~\ref{sec:tight}, we analyze the limitations of existing algorithms in achieving optimality. Finally, Section \ref{sec:conclusion} concludes the paper.

\section{Main Results}\label{sec:main}
In this section, we present our main theoretical results. We begin by stating a smoothness assumption on the objective function~$f$, which will be maintained throughout the paper.
\begin{assumption}\label{ass:1}
The function $f:\R^n\times \R^d\rightarrow \R$ is continuously differentiable, and its partial gradients $\nabla_{\x} f(\cdot,\cdot)$ and $\nabla_{\y} f(\cdot,\cdot)$ are $L_{\x}$ and $L_{\y}$-Lipschitz continuous on $\X\times \Y$. In particular, for all  $\x,\x' \in \X$ and $\y,\y'\in \Y$, we have 
		\[
		\begin{aligned}
			&\|\nabla_{\x}f(\x,\y)-\nabla_{\x}f(\x',\y')\| \leq L_{\x}(\|\x-\x'\|+\|\y-\y'\|),\\
			&\|\nabla_{\y}f(\x,\y)-\nabla_{\y}f(\x',\y')\| \leq L_{\y}(\|\x-\x'\|+\|\y-\y'\|). 
		\end{aligned}
		\]
 Without loss of generality (WLOG),  in the following, we let 
 $L_{\x}= L_{\y}=L\geq 1$.
  \end{assumption}
  \subsection{DS-OGDA}
We now introduce our proposed algorithm, Doubly Smoothed Optimistic GDA (DS-OGDA). 
The algorithm is built upon a regularized function $F:\R^n \times \R^d \times \R^n \times \R^d \rightarrow \R$, defined as follows:
\[
F(\x,\y,\z,\v)=f(\x,\y)+\frac{r_{\x}}{2}\|\x-\z\|^2-\frac{r_{\y}}{2}\|\y-\v\|^2, 
\] 
where $r_{\x}, r_{\y}\geq 0$ are smoothing parameters.

Let $\u = (\x, \y, \z, \v) \in \mathcal{X} \times \mathcal{Y} \times \mathbb{R}^n \times \mathbb{R}^d$. To describe the update directions in our algorithm, we define the operator 
\[
G_{\u}= [G_{\x}; G_{\y};G_{\z};G_{\v}] := [\nabla_{\x} F; -\nabla_{\y} F;\nabla_{\z} F; -\nabla_{\v} F],
\]
where each block corresponds to a descent or ascent direction in its respective variable. For simplicity, we use $G_{\x}^t$ (and similarly $G_{\y}^t, G_{\z}^t, G_{\v}^t$) to denote the evaluation at iteration $t>0$. 

Algorithm~\ref{alg:1} summarizes the DS-OGDA method, which synthesizes ideas from OGDA~\citep{mokhtari2020unified} and DS-GDA~\citep{zheng2023universal}, incorporating extrapolation steps to accelerate convergence in the convex-concave setting. The main distinction from DS-GDA, highlighted in red, is the inclusion of the gradient difference in both the $\x$ and $\y$ updates.

\begin{algorithm}
\caption{Doubly Smoothed Optimistic GDA (DS-OGDA)}\label{alg:1}
\begin{algorithmic}
\STATE{Initial  $\x^{-1}=\x^0,\y^{-1}=\y^0,\z^{-1}=\z^0=\x^0,\v^{-1}=\v^0=\y^0$, step-sizes $\eta_{\x} ,\eta_{\y}>0$, and smoothing parameters $0<\beta_{\x},\beta_{\y}<1$}
\FOR{$t=0,1,2, \ldots$}
\STATE{$\x^{t+1}=\proj_{\X}(\x^t-\eta_{\x}G_{\x}^t{\color{red!80!black}{ + \eta_{\x} \left(G_{\x}^{t-1}-G_{\x}^t\right)}})$;}
\STATE{$\y^{t+1}=\proj_{\Y}(\y^t-\eta_{\y} G_{\y}^t {\color{red!80!black}{+\eta_{\y} \left(G_{\y}^{t-1}-G_{\y}^t\right) }})$;}
\STATE{$\z^{t+1}=\z^t+ \beta_{\x}(\x^{t}-\z^t) $;}
\STATE{$\v^{t+1}=\v^t+ \beta_{\y}(\y^{t}-\v^t) $.}
\ENDFOR 
\end{algorithmic}
\end{algorithm}

\subsection{Convergence Results}
We present the main convergence results of DS-OGDA in this subsection, under a range of structural assumptions. We begin by introducing the standard stationarity measures that will be used throughout the rest of the paper.
\begin{definition}[Stationarity measures]\label{def:1}
Let $\epsilon\geq 0$ be given. 
\begin{enumerate}[label=(\roman*)] 
\setlength{\itemsep}{5pt}
\item The pair $(\x,\y)\in \mathcal{X}\times \mathcal{Y}$ is an $\epsilon$-saddle point {\rm($\epsilon$-SP)} if 
   \[ \max_{\y' \in \Y} f(\x,\y')-\min_{\x' \in \X} f(\x',\y) \leq \epsilon.\]
\item The pair $(\x,\y)\in \mathcal{X}\times \mathcal{Y}$ is an $\epsilon$-game stationary point {\rm ($\epsilon$-GS)} if 
			\[
\dist(\bz,\nabla_{\x}f(\x,\y)+\partial\iota_\X(\x))\leq \epsilon \quad \text{and} \quad \dist(\bz,-\nabla_{\y}f(\x,\y)+\partial\iota_\Y(\y))\leq \epsilon.
			\]
\item Given $r>L$, the point $ \x \in \mathcal{X} $ is an $\epsilon$-optimization stationary point {\rm ($\epsilon$-OS)} if  
			\[
			\left\|\prox_{\tfrac{1}{r}\cdot \phi}(\x)-{\x}\right\|\leq\epsilon,
   \] 
   where $\phi$ is the value function given by $\phi(\x) = \max_{\y' \in \Y} f(\x,\y') + \iota_{\X}(\x)$.
   \end{enumerate}
\end{definition}
\begin{remark}
In the C-C setting, we focus on convergence to SP, while for nonconvex minimax problems, we  consider both GS and OS. The relationships among these stationarity concepts have been extensively studied in prior works \citep{lin2020gradient,li2022nonsmooth}.
\end{remark} 
Next, we introduce the structural assumptions on Problem~\eqref{eq:prob} used in our convergence analysis. In particular, for the dual problem, we assume either concavity or that the associated function satisfies the K\L{} property.
\begin{assumption}[Structural assumptions for the dual problem]\label{ass:dual} One of the following conditions is satisfied by the dual problem:
\begin{enumerate}[label=(\roman*)] \setlength{\itemsep}{1pt}
\item {\rm (Concavity):} For all $\x \in \mathcal{X}$, the function $f(\x, \cdot): \mathbb{R}^d \rightarrow \mathbb{R}$ is concave.
\item {\rm(K\L{} property):} For all $\x \in \mathcal{X}$, the maximization problem $\max_{\y \in \Y}f(\x,\y)$
 has a nonempty solution set and a finite optimal value. Moreover, there exist constants $\tau_{\y} > 0$ and $\theta \in (0,1)$ such that 
		\[
				\left(\max_{\y' \in\Y} f(\x,\y') -f(\x,\y)\right)^{\theta} \leq \frac{1}{\tau_{\y}}\dist\left(\bz,-\nabla_{\y}f(\x,\y)+\partial\b1_\Y(\y)\right)
		\]  
        for all $\x \in \mathcal{X}$ and $\y \in \mathcal{Y}$.
\end{enumerate}
 \end{assumption}
The assumptions on the primal problem are analogous to those on the dual problem and are imposed symmetrically. They are presented in Assumption~\ref{ass:primal}.
 
 \begin{assumption}[Structural assumptions for the primal problem]\label {ass:primal} One of the following conditions is satisfied by the primal problem: 
 \begin{enumerate}[label=(\roman*)] \setlength{\itemsep}{1pt}
\item {\rm (Convexity):} For all $\y \in \mathcal{Y}$, the function $f(\cdot, \y): \mathbb{R}^n \rightarrow \mathbb{R}$ is convex.
 \item {\rm (K\L{} property):} 
 For all $\y \in \mathcal{Y}$, the minimization problem $\min_{\x \in \X}f(\x,\y)$ has a nonempty solution set and a finite optimal value. Moreover, there exist constants $\tau_{\x} > 0$ and $\theta \in (0,1)$ such that 
		\[
				\left(f(\x,\y)-\min_{\x' \in\X} f(\x',\y)  \right)^{\theta} \leq \frac{1}{\tau_{\x}}\dist(\bz,\nabla_{\x}f(\x,\y)+\partial\b1_\X(\x))
		\]   
        for all $\x \in \mathcal{X}$ and $\y \in \mathcal{Y}$.
  \end{enumerate}
 \end{assumption}
For C-C problems, an additional assumption is required when the constraint sets are only closed and convex.
 \begin{assumption}[Existence of saddle point]\label{ass:5}
The solution set 
\[
\U_{\textrm{pd}}^\star=\left\{[\x;\y] \in \R^{n+d}:  (\x,\y) \text{ is a saddle point of problem \eqref{eq:prob}} \right\}
\]
is nonempty.
\end{assumption}

\begin{remark}
(i) The K\L{} property in Assumptions \ref{ass:dual}(ii) and \ref{ass:primal}(ii) requires a uniform K\L{} exponent $\theta \in (0,1)$ and constants $\tau_{\y}, \tau_{\x} > 0$ for all $\x \in \mathcal{X}$ and $\y \in \mathcal{Y}$. This captures the global landscape of $f(\x, \cdot)$ and $f(\cdot, \y)$, respectively. The K\L{} property is a widely adopted assumption in the literature for establishing the convergence of the sequence of iterates \citep{bolte2007lojasiewicz, attouch2010proximal, attouch2013convergence, chen2021proximal}. In our context, the one-sided K\L{} property is crucial for balancing the primal and dual updates, thereby ensuring the convergence of the DS-OGDA algorithm. This perspective was first introduced by \citet{li2022nonsmooth} and further explored in \citet{zheng2023universal}. 
(ii) Assumption~\ref{ass:5} can be omitted when both constraint sets $\mathcal{X}$ and $\mathcal{Y}$ are compact. In this case, the existence of a saddle point is guaranteed by the minimax theorem under the convex–concave structure. 
\end{remark}
We are now ready to demonstrate the universality of DS-OGDA. WLOG, we consider the case when $\text{diam}(\mathcal{X}), \text{diam}(\mathcal{Y}) \geq 1$.

    \begin{theorem}[Universal  complexity results of DS-OGDA]\label{theorem:universal}
    Suppose that at least one of the following conditions holds: Assumption~\ref{ass:dual}(i), Assumption~\ref{ass:dual}(ii), Assumption~\ref{ass:primal}(i), or Assumption~\ref{ass:primal}(ii).  For any $\underline{\varepsilon} > 0$ and any integer $T > 1$, if the K\L{} exponent $\theta$ of the primal or dual problems satisfies $\theta \in [\underline{\varepsilon}, 1]$, we can select a set of symmetric parameters with $0 < r_{\x} = r_{\y} = \mathcal{O}(1)$, $0 < \eta_{\x} = \eta_{\y} = \mathcal{O}(1)$, and $0 < \beta_{\x} = \beta_{\y} \leq \mathcal{O}(T^{-1/2})$. Then, there exists an index $t \in \{0, 1, \ldots, T-1\}$ such that: 
    \begin{enumerate}[label=(\roman*)] \setlength{\itemsep}{1pt} \item {\rm (General)} The point $(\x^{t+1}, \y^{t+1})$ is an $\mathcal{O}(T^{-1/4})$-GS of Problem \eqref{eq:prob}. 
    
    \item {\rm (C-C Problems)} If both Assumptions \ref{ass:dual}(i) and \ref{ass:primal}(i) hold, then $(\x^{t+1}, \y^{t+1})$ is an $\mathcal{O}(T^{-1/2})$-SP of Problem \eqref{eq:prob}. \end{enumerate} \end{theorem}

When additional structural properties, such as convex–concavity, are verified, the optimal convergence rate can be achieved. In such cases, the boundedness assumptions on the constraint sets are no longer required.
\begin{theorem}[Iteration complexity for C-C problems]\label{theorem:cc} 
    Suppose $\X$ and $\Y$ are only closed and convex sets, and the function $f$ is convex-concave, i.e., both Assumptions \ref{ass:dual}(i) and \ref{ass:primal}(i) hold. We further assume the existence of a saddle point, i.e., Assumption \ref{ass:5} holds. For any integer $T> 1$, if we set $0\leq r_{\x}=r_{\y} \leq\mO(T^{-1})$, $0< \eta_{\x}=\eta_{\y}\leq \tfrac{1}{7(L +r_{\x})}  $, and $0\leq  \beta_{\x}=\beta_{\y}=r_{\x}\eta_{\x} \leq \mO(T^{-1})$, then there exists an index $t\in\{0,1,\ldots, T-1\}$ such that
 $\left( \x^{t+1},  \y ^{t+1} \right)$ is an $\mO(T^{-1})$-SP of Problem \eqref{eq:prob}.
\end{theorem}
The above theorem illustrates that verifying additional structural conditions—such as convex–concavity allows more flexibility in parameter choices.
This advantage extends to nonconvex minimax problems as well. Specifically, for NC-C and C-NC problems, it suffices for only one of the smoothing step sizes, $\beta_{\x}$ or $\beta_{\y}$, to depend on the total number of iterations $T$. Furthermore, in NC-K\L{} and K\L{}-NC problems, such prior information not only enables this flexibility, but also leads to improved iteration complexity.

\begin{theorem}[Iteration complexity for NC-(N)C problems]\label{theorem:2} 
Suppose $\Y$ is compact and $\X$ is a closed and convex set.
For any integer $T > 1$, we can select a set of parameters with $0 < r_{\x}, r_{\y}, \eta_{\x}, \eta_{\y}, \beta_{\x}, \beta_{\y} = \mO(1)$, and there exists an index $t \in \{0, 1, \ldots, T-1\}$ such that the following hold:
\begin{enumerate}[label=(\roman*)] \setlength{\itemsep}{1pt}
\item  {\rm (Concave dual)}  Suppose Assumption \ref {ass:dual}(i) holds. If $\max_{\y\in \Y} f(\cdot, \y)$ is bounded below on $\X$ and $\beta_{\x} \leq \mO(T^{-1/2})$, then $(\x^{t+1},\y^{t+1})$ is an $\mO(T^{-1/4})$-GS and $\z^{t+1}$ is an $\mO(T^{-1/4})$-OS of Problem \eqref{eq:prob}. 
\item {\rm (K\L{} Property)} Suppose Assumption \ref {ass:dual}(ii) holds with $\theta\in (0,1)$. If $\X$ is compact and $\beta_{\x}\leq \mO(T^{- \left(2\theta-1\right)_{+}/2\theta})$, then $(\x^{t+1},\y^{t+1})$ is an $\mO(T^{-1/\left((4\theta-2)_++2\right)})$-GS and $\z^{t+1}$ is an $\mO(T^{-1/\left((4\theta-2)_++2\right)})$-OS of Problem \eqref{eq:prob}. 
\end{enumerate}
\end{theorem}
Similar results hold for (NC)-NC problems, as stated in the following corollary. 
\begin{corollary}[Iteration complexity for (N)C-NC problems]\label{col:1} Suppose $\X$ is compact and $\Y$ is a closed and convex set. For any integer $T > 1$, we can select a set of  parameters with $0 < r_{\x}, r_{\y}, \eta_{\x}, \eta_{\y}, \beta_{\x},\beta_{\y} = \mO(1)$, and there exists an index $t \in \{0,1,  \ldots, T-1\}$ such that the following hold:
\begin{enumerate}[label=(\roman*)] \setlength{\itemsep}{1pt}
\item {\rm (Convex primal)} Suppose Assumption \ref{ass:primal}(i) holds. If $\min_{\x \in \X} f(\x,\cdot)$ is bounded below on $\Y$ and $\beta_{\y} \leq \mO(T^{-1/2})$, then $(\x^{t+1},\y^{t+1})$ is an $\mO(T^{-1/4})$-GS of Problem \eqref{eq:prob}. 
\item {\rm (K\L\ property)} Suppose Assumption \ref {ass:primal}(ii) holds with $\theta\in (0,1)$. If $\Y$ is compact and $\beta_{\y}\leq \mO(T^{- \left(2\theta-1\right)_{+}/2\theta})$ , then $(\x^{t+1},\y^{t+1})$ is an $\mO(T^{-1/\left((4\theta-2)_++2\right)})$-GS of Problem \eqref{eq:prob}. 
\end{enumerate}
\end{corollary}

\section{Proof of Main Results}\label{sec:proof}
In this section, we provide detailed proofs of our main convergence results. Specifically, we present the proof for C-C problems in Section~\ref{sec:proof_cc} (see Theorem~\ref{theorem:cc}). We then provide the proofs for NC-C, C-NC, NC-K\L{}, K\L{}-NC minimax problems in Section~\ref{sec:proof_nc} (see Theorem~\ref{theorem:2} and Corollary~\ref{col:1}). Finally, we establish the universal convergence result in Section~\ref{sec:proof_universal} (see Theorem~\ref{theorem:universal}).

 \subsection{Proof for Convex-Concave Minimax Problems}\label{sec:proof_cc}
The convergence of OGDA for C-C problems is well established, achieving an optimal iteration complexity of $\mathcal{O}(T^{-1})$ \citep{mokhtari2020convergence, lu2022s}. Given the symmetric structure of C-C problems, we adopt symmetric parameter choices: $r\coloneqq r_{\x}=r_{\y}$, $\eta \coloneqq \eta_{\x}=\eta_{\y}$, and $\beta\coloneqq \beta_{\x}=\beta_{\y}$. When $r = 0$ and $\beta = 0$, our algorithm reduces to vanilla OGDA, whose convergence is proven by viewing it as an inexact Proximal Point Method (PPM). With the introduction of smoothing parameters ($r > 0$ and $\beta > 0$), we view our proposed DS-OGDA as an inexact PPM applied to the regularized problem $F$. Specifically, since $F$ is jointly convex-concave in $(\x, \z)$ and $(\y, \v)$, as demonstrated in Lemma~\ref{lem:cc}, we can leverage the analysis framework from \citet{mokhtari2020convergence}.

However, directly applying this framework is insufficient. Both $\z$ and $\v$ in our algorithm follow standard gradient descent and ascent updates, respectively, which leads to unbounded inexactness that complicates the convergence analysis. To address this, we conduct a refined analysis that leverages the structure of both the algorithm and the underlying problem, enabling us to jointly control the inexactness across all variables. This is formalized in Proposition~\ref{proposition:cc}, where we establish the key bounds required for convergence.

\begin{lemma}[Joint convex-concavity of $F$]
    \label{lem:cc}
Suppose that Assumptions~\ref{ass:dual}(i) and~\ref{ass:primal} (i) hold, and let $r > 0$. Then, the function $F$ is jointly convex in $(\x, \z)$ and jointly concave in $(\y, \v)$. 
 \end{lemma}
 \begin{proof} 
We first establish the joint convexity of $F$ in $(\x, \z)$. The joint concavity in $(\y, \v)$ follows analogously by leveraging the concavity of $f$ in $\y$.
 
 Let $\lambda \in [0,1]$, and consider two points $(\x, \y, \z, \v)$ and $(\x', \y, \z', \v)$ in the domain of $F$. Then, we have  
    \begin{align*}
        &F(\lambda \x+(1-\lambda) \x',\y, \lambda \z+(1-\lambda)\z',\v ) \\
        =\ & f(\lambda \x+(1-\lambda) \x',\y)+\frac{r }{2}\|\lambda \left(\x-\z\right)+(1-\lambda) \left(\x'-\z'\right)\|^2-\frac{r }{2}\|\y-\v\|^2 \\
       \leq\ &  \lambda f(\x,\y)+(1-\lambda) f(\x',\y)+\frac{r }{2}\left(\lambda  \|\x-\z\|^2+(1-\lambda) \|\x'-\z'\|^2\right)-\frac{r }{2}\|\y-\v\|^2  \\  
       =\ & \lambda F(\x,\y,\z,\v)+(1-\lambda) F(\x',\y,\z',\v), \notag
    \end{align*}
    where the inequality follows from the convexity of $f(\cdot, \y)$ and $\|\cdot\|^2$. 
 \end{proof}

 We are now prepared to establish the convergence of the DS-OGDA algorithm. To facilitate a concise presentation, we introduce the following notations: Let $\u_{\textrm{pd}} = [\x; \y]$ denote the primal-dual pair, and $\u_{\textrm{ex}} = [\z; \v]$ represent the variables generated by the additional smoothing steps. The combined variable is denoted by $\u = [\u_{\textrm{pd}}; \u_{\textrm{ex}}]$, and its feasible set is defined as $\U$. Similarly, the feasible set for the primal-dual pair is given by $\U_{\textrm{pd}} = \X \times \Y$.
 With these notations, we introduce a controllable upper bound for the stationarity measure as defined in Definition~\ref{def:1}(i).

 \begin{lemma} \label{lem:fund_cc}
  Suppose that Assumptions \ref {ass:dual}(i) and \ref{ass:primal}(i), and \ref{ass:5} hold. For any integer $T>1$, let the sequence $\{\u^t\}_{t=1}^{T}$ be generated by DS-OGDA. Then, for any $\u \in \U$, the following inequality holds:
  \vspace{-2mm}
    \[
     \min_{t\in \{0,1,\ldots, T-1\}} \left\{\max_{\y\in \Y} f(\x^{t+1},\y)-\min_{\x\in \X} f(\x,\y^{t+1})\right\} \leq \frac{1}{T} \sum_{t=0}^{T-1} \langle G_{\u}^{t+1},\u^{t+1}-\u\rangle.
    \]
\end{lemma}

\begin{proof}
 First, we relate the stationarity measure of the original problem $f$ to that of the regularized function $F$ as follows:
    \begin{align}\label{eq:alter_cc}
&\max_{\y\in \Y} f(\x^t,\y)-\min_{\x\in \X} f(\x,\y^t) \notag\\
\leq\ & \max_{\y\in \Y} \left(f(\x^t,\y)+\frac{r}{2}\|\x^t-\z^t\|^2\right)-\min_{\x\in \X} \left(f(\x,\y^t)-\frac{r}{2}\|\y^t-\v^t\|^2\right) \notag\\
=\ & \max_{\u\in \U} F(\x^t,\y,\z^t,\v)- F(\x,\y^t,\z,\v^t).
\end{align}
Then, we can utilize the proof from \citet[Proposition 1]{mokhtari2020convergence}, applied to the regularized function $F$, to complete the remaining parts of our proof.
\end{proof}

Leveraging the computable upper bound of the stationarity measure established in Lemma~\ref{lem:fund_cc}, we aim to derive convergence results by controlling this bound. To this end, we reinterpret the DS-OGDA algorithm as an inexact Proximal Point Method (PPM) applied to the operator  $G$. Specifically, by setting $\beta = \eta r$, the update rule of DS-OGDA can be rewritten as $\u^{t+1}=\proj_{\U} (\u^t-\eta G_{\u}^{t+1}+\meps^{t} )$, 
 where the inexact error $ \meps^{t}$ can be computed as 
\begin{align}
    \label{eq:error}
    \meps^{t}
=\eta \begin{pmatrix}
G^{t+1}_{\u_{\textrm{pd}}}-2G^{t}_{\u_{\textrm{pd}}}+G^{t-1}_{\u_{\textrm{pd}}} \\
G^{t+1}_{\u_{\textrm{ex}}}-G^{t}_{\u_{\textrm{ex}}}
\end{pmatrix},
\end{align}  
where $G_{\u_{\textrm{pd}}}^t$ represents the components of 
 $G_{\u}^t$ associated with $\u_{\textrm{pd}}^t$, and $G^{t}_{\u_{\textrm{ex}}}$ denotes the components corresponding to $\u_{\textrm{ex}}^t$.

Clearly, the error  $\meps^t$  decomposes into two parts: The first arises from extrapolation in the primal-dual component, and the second from the smoothing steps. This structural decomposition is reflected in the tight upper bound given in the first part of Proposition~\ref{proposition:cc}. 
As noted earlier, the main challenge lies in ensuring the boundedness of the iterates, which is not directly implied by the bound in Proposition~\ref{proposition:cc}(i). We address this issue by leveraging a coupling between the proximal variables and the primal-dual pair, wherein the optimal solution satisfies $\u_{\mathrm{ex}}^\star = \u_{\mathrm{pd}}^\star $. The resulting explicit bound on the iterates is established in the second part of Proposition~\ref{proposition:cc}.

\begin{proposition} \label{proposition:cc}
Suppose that Assumptions \ref{ass:dual}(i), \ref{ass:primal}(i), and \ref{ass:5} are satisfied. For any integer $T>1$, let $\{\u^t\}_{t=1}^{T}$ be the iterates generated by DS-OGDA. If $r\geq 0$ and $0<\eta,\beta=\eta r \leq 1$, then the following hold: 
\begin{enumerate}[label=(\roman*)] \setlength{\itemsep}{1pt}
\item {\rm{(Upper bound).}} For any $\u \in \U$ and $t\geq 0$, we have
\begin{align*}
       \langle G_{\u}^{t+1},\u^{t+1}-\u\rangle 
    \leq\  & \frac{1}{2\eta}\|\u-\u^t\|^2-\frac{1}{2\eta}\|\u-\u^{t+1}\|^2-\frac{1}{2\eta}\|\u^{t+1}-\u^t\|^2+\frac{L+r}{2}\|\u_{\textrm{pd}}^{t+1}-\u_{\textrm{pd}}^t\|^2 \notag \\
   & +\langle G^{t+1}_{\u_{\textrm{pd}}}-G^{t}_{\u_{\textrm{pd}}},\u_{\textrm{pd}}^{t+1}-\u_{\textrm{pd}} \rangle -\langle G^{t}_{\u_{\textrm{pd}}}-G^{t-1}_{\u_{\textrm{pd}}},\u_{\textrm{pd}}^{t}- \u_{\textrm{pd}}\rangle  \notag \\
   &+\langle G^{t+1}_{\u_{\textrm{ex}}}-G^{t}_{\u_{\textrm{ex}}},  \u_{\textrm{ex}}^{t+1}-\u_{\textrm{ex}}\rangle + 3\left(L+r\right) \|\u^t-\u^{t-1}\|^2. 
\end{align*}
\item {\rm{(Boundedness of iterates).}} Let $0<\eta\leq \tfrac{1}{7(L+r)}$, then the iterates $\{\u^t\}_{t=1}^{T}$ stay within the compact set $\D$ defined as
\[
\D\coloneqq \left\{\u\in \U \mid  \|\u -\u^\star\|^2\leq  3\|\u^0-\u^\star\|^2 \right\},
\]
where $\u^\star=[\x^\star;\y^\star;\x^\star;\y^\star]$ and $\left(\x^\star,\y^\star\right)$ is the saddle point of problem \eqref{eq:prob}.  
\end{enumerate}
 \end{proposition}
\begin{proof}[Proof of Proposition \ref{proposition:cc}]
We begin with the following lemma, which provides a sharp upper bound for the inexact PPM. This result was originally established in~\citet{mokhtari2020convergence}, and we extend it here to handle constrained settings.

\begin{lemma}[Inexact PPM]\label{lem:ippm} Suppose that Assumptions \ref {ass:dual}(i) and \ref{ass:primal}(i) hold. For any integer $T>1$,
we consider the sequence of iterates $\{\u^t\}_{t=1}^{T}$ generated by the inexact PPM: $\u^{t+1}=\proj_{\U}(\u^t  -\eta G_{\u}^{t+1}+\meps^t)$ with $\meps^t\in \R^{n}\times \R^d\times \R^n\times \R^d$. Then, for any $t\geq 0$, the following holds for all $\u\in \U$:
     \[
    \begin{aligned}
       \langle G_{\u}^{t+1},\u^{t+1}-\u\rangle\leq\ & \frac{1}{2\eta}\|\u -\u^t\|^2-\frac{1}{2\eta}\|\u-\u^{t+1}\|^2-\frac{1}{2\eta}\|\u^{t+1}-\u^t\|^2 +\frac{1}{\eta}\langle \meps^t,\u^{t+1}-\u \rangle. 
    \end{aligned}
    \]
\end{lemma}
(i) By substituting the algorithm-tailored $\meps^t$ (see \eqref{eq:error}) into Lemma \ref{lem:ippm}, we obtain
\begin{align}
\label{eq:cc_0}
       \langle G_{\u}^{t+1},\u^{t+1}-\u\rangle  
       \leq \ & \frac{1}{2\eta}\|\u-\u^t\|^2-\frac{1}{2\eta}\|\u-\u^{t+1}\|^2-\frac{1}{2\eta}\|\u^{t+1}-\u^t\|^2 \notag\\
       &+ \underbrace{\langle  G^{t+1}_{\u_{\textrm{pd}}}-2G^{t}_{\u_{\textrm{pd}}}+G^{t-1}_{\u_{\textrm{pd}}},\u_{\textrm{pd}}^{t+1}-\u_{\textrm{pd}}\rangle }_\text{\O1} +\langle G^{t+1}_{\u_{\textrm{ex}}}-G^{t}_{\u_{\textrm{ex}}}, \u_{\textrm{ex}}^{t+1}-\u_{\textrm{ex}}\rangle.
    \end{align}
The term \O1 arises from the OGDA update. As discussed in \citet{mokhtari2020convergence}, we have  
\begin{align}\label{eq:cc_1}
   \text{\O1}=\ & \langle G^{t+1}_{\u_{\textrm{pd}}}-G^{t}_{\u_{\textrm{pd}}},\u_{\textrm{pd}}^{t+1}-\u_{\textrm{pd}}\rangle - \langle G^{t}_{\u_{\textrm{pd}}}-G^{t-1}_{\u_{\textrm{pd}}},\u_{\textrm{pd}}^{t}-\u_{\textrm{pd}}\rangle   +\langle G^{t}_{\u_{\textrm{pd}}}-G^{t-1}_{\u_{\textrm{pd}}},\u_{\textrm{pd}}^{t}-\u_{\textrm{pd}}^{t+1}\rangle \notag\\
   \leq\ &  \langle G^{t+1}_{\u_{\textrm{pd}}}-G^{t}_{\u_{\textrm{pd}}},\u_{\textrm{pd}}^{t+1}-\u_{\textrm{pd}}\rangle - \langle G^{t}_{\u_{\textrm{pd}}}-G^{t-1}_{\u_{\textrm{pd}}},\u_{\textrm{pd}}^{t}-\u_{\textrm{pd}}\rangle  \notag\\
   &+3\left(L+r\right) \|\u^t-\u^{t-1}\|^2+ \frac{L+r}{2} \|\u_{\textrm{pd}}^{t}-\u_{\textrm{pd}}^{t+1}\|^2, 
\end{align}
where the inequality follows from 
\begin{align} 
    &\langle G^{t}_{\u_{\textrm{pd}}}-G^{t-1}_{\u_{\textrm{pd}}},\u_{\textrm{pd}}^{t}-\u_{\textrm{pd}}^{t+1}\rangle \notag\\
    =\ & \langle \nabla_{\x} F(\x^t,\y^t,\z^t,\v^t)-\nabla_{\x} F(\x^{t-1},\y^{t-1},\z^{t-1},\v^{t-1}), \x^t-\x^{t+1}\rangle \notag \\
    &+\langle -\nabla_{\y} F(\x^t,\y^t,\z^t,\v^t)+\nabla_{\y} F(\x^{t-1},\y^{t-1},\z^{t-1},\v^{t-1}), \y^t-\y^{t+1}\rangle \notag\\
    \leq\ & \left(L+r\right)\|\x^t-\x^{t+1}\| \left(\|\x^t-\x^{t-1}\|+\|\y^t-\y^{t-1}\|+\|\z^t-\z^{t-1}\|\right) \notag \\
    &+\left(L+r\right)\|\y^t-\y^{t+1}\| \left(\|\x^t-\x^{t-1}\|+\|\y^t-\y^{t-1}\|+\|\v^t-\v^{t-1}\|\right) \label{eq:bd_cs}\\
    \leq\ & \frac{L+r}{2}\|\x^t-\x^{t+1}\|^2+\frac{3\left(L+r\right)}{2}\left(\|\u_{\textrm{pd}}^{t}-\u_{\textrm{pd}}^{t-1}\|^2+\|\z^t-\z^{t-1}\|^2\right) \notag \\
    &+\frac{L+r}{2}\|\y^t-\y^{t+1}\|^2+\frac{3\left(L+r\right)}{2}\left(\|\u_{\textrm{pd}}^{t}-\u_{\textrm{pd}}^{t-1}\|^2+\|\v^t-\v^{t-1}\|^2\right) \notag \\
    \leq\ & 3\left(L+r\right) \|\u^t-\u^{t-1}\|^2+ \frac{L+r}{2} \|\u_{\textrm{pd}}^{t}-\u_{\textrm{pd}}^{t+1}\|^2 \label{eq:detailed_bd}.
\end{align}
Here, the first inequality follows from Assumption \ref{ass:1}, and the second from AM-GM and Cauchy-Schwarz inequalities. Substituting \eqref{eq:cc_1} into \eqref{eq:cc_0} completes the proof. 

(ii) Now, we move to the second result. Using the definitions of $G$ and $\u$, we expand the stationary measure studied in Lemma \ref{lem:fund_cc} as follows:
\begin{align}
    \label{eq:cc_lb}
        \langle G_{\u}^{t+1},\u^{t+1}-\u\rangle =\ & \langle \nabla_{\x} f(\x^{t+1},\y^{t+1})  ,\x^{t+1}-\x \rangle+\langle   -\nabla_{\y} f(\x^{t+1},\y^{t+1}), \y^{t+1}-\y\rangle \notag\\
        &-\langle G^{t+1}_{\u_{\textrm{ex}}},\u_{\textrm{pd}}^{t+1}-\u_{\textrm{pd}}\rangle + \langle G^{t+1}_{\u_{\textrm{ex}}},\u_{\textrm{ex}}^{t+1}-\u_{\textrm{ex}} \rangle.
    \end{align} 
Plugging \eqref{eq:cc_lb} into Proposition \ref{proposition:cc}(i) and rearranging it, we obtain
\begin{align}
    \label{eq:cc_ub} 
        & \langle \nabla_{\x} f(\x^{t+1},\y^{t+1}) ,\x^{t+1}-\x \rangle+\langle   -\nabla_{\y} f(\x^{t+1},\y^{t+1}) , \y^{t+1}-\y\rangle \notag\\
        \leq\ & \frac{1}{2\eta}\|\u-\u^t\|^2-\frac{1}{2\eta}\|\u-\u^{t+1}\|^2+  3\left(L+r\right) \|\u^t-\u^{t-1}\|^2-\frac{1}{2\eta}\|\u^{t+1}-\u^t\|^2 \notag\\
   &+\frac{L+r}{2}\|\u_{\textrm{pd}}^{t+1}-\u_{\textrm{pd}}^t\|^2 + \langle G^{t+1}_{\u_{\textrm{pd}}}-G^{t}_{\u_{\textrm{pd}}},\u_{\textrm{pd}}^{t+1}-\u_{\textrm{pd}}\rangle   - \langle G^{t}_{\u_{\textrm{pd}}}-G^{t-1}_{\u_{\textrm{pd}}},\u_{\textrm{pd}}^{t}-\u_{\textrm{pd}}\rangle \notag\\
   & + \underbrace{\langle G^{t+1}_{\u_{\textrm{ex}}},\u_{\textrm{pd}}^{t+1}-\u_{\textrm{pd}}\rangle-\langle G^{t}_{\u_{\textrm{ex}}}, \u_{\textrm{ex}}^{t+1}-\u_{\textrm{ex}}\rangle }_\text{\O2}. 
\end{align}
Next, we aim to control the additional approximation error terms introduced by the smoothing steps, specifically the term \O2. The core idea parallels the technique used in bounding \O1. We decompose the error into two components: One forms a telescoping series in which most terms cancel out, and the other admits a direct bound. 
To facilitate this, we highlight a key observation: $\u_{\textrm{pd}}^\star=\u_{\textrm{ex}}^\star$. Substituting $\u=\u^\star$ in term \O2 and using the update of $\u_{\textrm{ex}}^{t+1}$, we have 
\begin{align}\label{eq:cc_3}
    \text{\O2}
    =\ & \langle G^{t+1}_{\u_{\textrm{ex}}},\u_{\textrm{pd}}^{t+1}-\u_{\textrm{pd}}^\star\rangle-\langle G^{t}_{\u_{\textrm{ex}}}, \u_{\textrm{pd}}^{t}-\u_{\textrm{pd}}^\star\rangle+\langle G^{t}_{\u_{\textrm{ex}}}, \u_{\textrm{pd}}^{t}-\u_{\textrm{ex}}^{t+1}\rangle \notag\\
    =\ &\langle G^{t+1}_{\u_{\textrm{ex}}},\u_{\textrm{pd}}^{t+1}-\u_{\textrm{pd}}^\star\rangle-\langle G^{t}_{\u_{\textrm{ex}}}, \u_{\textrm{pd}}^{t}-\u_{\textrm{pd}}^\star\rangle -\frac{r\left(1-\beta\right)}{\beta^2}\|\u_{\textrm{ex}}^{t}-\u_{\textrm{ex}}^{t+1}\|^2.
\end{align}
By setting $\u=\u^\star$ in \eqref{eq:cc_ub} and incorporating \eqref{eq:cc_3}, we obtain
\begin{align}\label{eq:cc_main} 
& \langle \nabla_{\x} f(\x^{t+1},\y^{t+1}) ,\x^{t+1}-\x^\star \rangle+\langle   -\nabla_{\y} f(\x^{t+1},\y^{t+1}) , \y^{t+1}-\y^\star\rangle \notag\\
\leq\ & \frac{1}{2\eta}\|\u^\star-\u^t\|^2-\frac{1}{2\eta}\|\u^\star-\u^{t+1}\|^2+  3\left(L+r\right) \|\u^t-\u^{t-1}\|^2 -\left(\frac{1}{2\eta}-\frac{L+r}{2}\right)\|\u^{t+1}-\u^t\|^2  \notag \\
& -\frac{r\left(1-\beta\right)}{\beta^2}\|\u_{\textrm{ex}}^{t}-\u_{\textrm{ex}}^{t+1}\|^2  + \langle G^{t+1}_{\u_{\textrm{pd}}}-G^{t}_{\u_{\textrm{pd}}},\u_{\textrm{pd}}^{t+1}-\u_{\textrm{pd}}^\star\rangle - \langle G^{t}_{\u_{\textrm{pd}}}-G^{t-1}_{\u_{\textrm{pd}}},\u_{\textrm{pd}}^{t}-\u_{\textrm{pd}}^\star\rangle \notag \\
   &+\langle G^{t+1}_{\u_{\textrm{ex}}},\u_{\textrm{pd}}^{t+1}-\u_{\textrm{pd}}^\star\rangle-\langle G^{t}_{\u_{\textrm{ex}}}, \u_{\textrm{pd}}^{t}-\u_{\textrm{pd}}^\star\rangle \notag\\
   \leq\ &\frac{1}{2\eta}\|\u^\star-\u^t\|^2-\frac{1}{2\eta}\|\u^\star-\u^{t+1}\|^2+  3\left(L+r\right) \|\u^t-\u^{t-1}\|^2 -3\left(L+r\right)\|\u^{t+1}-\u^t\|^2  \notag \\
& -\frac{r\left(1-\beta\right)}{\beta^2}\|\u_{\textrm{ex}}^{t}-\u_{\textrm{ex}}^{t+1}\|^2  + \langle G^{t+1}_{\u_{\textrm{pd}}}-G^{t}_{\u_{\textrm{pd}}},\u_{\textrm{pd}}^{t+1}-\u_{\textrm{pd}}^\star\rangle - \langle G^{t}_{\u_{\textrm{pd}}}-G^{t-1}_{\u_{\textrm{pd}}},\u_{\textrm{pd}}^{t}-\u_{\textrm{pd}}^\star\rangle \notag \\
   &+\langle G^{t+1}_{\u_{\textrm{ex}}},\u_{\textrm{pd}}^{t+1}-\u_{\textrm{pd}}^\star\rangle-\langle G^{t}_{\u_{\textrm{ex}}}, \u_{\textrm{pd}}^{t}-\u_{\textrm{pd}}^\star\rangle. 
\end{align}
The  second inequality follows from $\eta\leq \tfrac{1}{7\left(L+r\right)}$, which ensures that $\tfrac{1}{2\eta}-\tfrac{L+r}{2} \geq3\left(L+r\right)$. 
By Assumptions~\ref{ass:dual}(i) and~\ref{ass:primal}(i), the left-hand side of~\eqref{eq:cc_main} is nonnegative. Summing both sides of~\eqref{eq:cc_main} from \( t = 0 \) to \( T - 1 \), and noting that \( \u_{\mathrm{pd}}^0 = \u_{\mathrm{ex}}^0 \), we obtain:
\begin{align}\label{eq:cc_bd} 
    0\leq\ &\frac{1}{2\eta}\|\u^\star -\u^0\|^2-\frac{1}{2\eta}\|\u^\star -\u^{T}\|^2   -\frac{r\left(1-\beta\right)}{\beta^2}\sum_{t=0}^{T-1}\|\u_{\textrm{ex}}^{t+1}-\u_{\textrm{ex}}^t\|^2  \notag \\
    & -3\left(L+r\right)\|\u^T-\u^{T-1}\|^2  +\underbrace{ \langle G^{T}_{\u_{\textrm{pd}}}-G^{T-1}_{\u_{\textrm{pd}}},\u_{\textrm{pd}}^{T}-\u_{\textrm{pd}}^\star\rangle + \langle G^{T}_{\u_{\textrm{ex}}},\u_{\textrm{pd}}^{T}-\u_{\textrm{pd}}^\star\rangle}_{\text{\O3}}.  
    \end{align}
  To proceed, we denote the gradient operator of $f$ as $H:\X\times \Y\rightarrow \R^n\times \R^d$, defined as $H=[\nabla_{\x} f; -\nabla_{\y} f]$. Let $H^t_{\u_{\textrm{pd}}}$ denote $H$ evaluated at  $\u_{\textrm{pd}}^t$. 

 Then, we proceed to bound term \O3: 
\begin{align}\label{eq:bd_3}
    \text{\O3}=\ & \langle H^{T}_{\u_{\textrm{pd}}}- H^{T-1}_{\u_{\textrm{pd}}}, \u_{\textrm{pd}}^{T}-\u_{\textrm{pd}}^\star\rangle + \langle G^{T-1}_{\u_{\textrm{ex}}}, \u_{\textrm{pd}}^{T}-\u_{\textrm{pd}}^\star \rangle \notag  \\
    \leq\ & 2L\|\u_{\textrm{pd}}^{T}-\u_{\textrm{pd}}^{T-1}\|^2 +\frac{L}{2}\|\u_{\textrm{pd}}^T-\u_{\textrm{pd}}^\star\|^2  +\frac{r}{2}\left(2\beta \| \u_{\textrm{pd}}^{T-1}-\u_{\textrm{ex}}^{T-1}\|^2+\frac{\|\u_{\textrm{pd}}^{T}-\u_{\textrm{pd}}^\star\|^2}{2\beta}\right) \notag\\
    =\ & 2L\|\u_{\textrm{pd}}^{T}-\u_{\textrm{pd}}^{T-1}\|^2+\left(\frac{L}{2} +\frac{1}{4\eta}\right)\|\u_{\textrm{pd}}^{T}-\u_{\textrm{pd}}^\star\|^2+\frac{r}{\beta}\|\u_{\textrm{ex}}^{T}-\u_{\textrm{ex}}^{T-1}\|^2. 
\end{align}
Here, the first equality follows from the decomposition $G_{\u_{\textrm{pd}}}=H_{\u_{\textrm{pd}}}-G_{\u_{\textrm{ex}}}$. The inequality is derived from the definition of  $G^{T-1}_{\u_{\textrm{ex}}}$, along with  the following bound: 
\begin{align*}
     \langle H^{T}_{\u_{\textrm{pd}}}- H^{T-1}_{\u_{\textrm{pd}}}, \u_{\textrm{pd}}^{T}-\u_{\textrm{pd}}^\star\rangle 
     \leq\ & L\left( \|\x^T-\x^\star\|+\|\y^T-\y^\star\|\right)\left(\|\x^T-\x^{T-1}\|+\|\y^T-\y^{T-1}\| \right)\\
     \leq\ & L\left(\|\x^T-\x^{T-1}\|+\|\y^T-\y^{T-1}\| \right)^2+\frac{L}{4} \left( \|\x^T-\x^\star\|+\|\y^T-\y^\star\|\right)^2\\
     \leq\ & 2L\|\u_{\textrm{pd}}^T-\u_{\textrm{pd}}^{T-1}\|^2+\frac{L}{2}\|\u_{\textrm{pd}}^{T}-\u_{\textrm{pd}}^\star\|^2,
\end{align*}
where the first inequality follows from Assumption~\ref{ass:1} and the Cauchy-Schwarz inequality, and the second from the AM-GM inequality.  The final equality in \eqref{eq:bd_3} follows from the definition of $G_{\u_{\textrm{ex}}}^T$ and the condition $\beta=\eta r$. 

Incorporating \eqref{eq:cc_bd} and \eqref{eq:bd_3}, we get 
 \begin{align*}
    0\leq\  &\frac{1}{2\eta}\|\u^\star -\u^0\|^2-\left(\frac{1}{4\eta}-\frac{L}{2}\right)\|\u^\star -\u^{T}\|^2- \left(L+r\right)\|\u^T-\u^{T-1}\|^2 \notag\\
    &-\frac{r\left(1-\beta\right)}{\beta^2}\sum_{t=0}^{T-2} \|\u_{\textrm{ex}}^t-\u_{\textrm{ex}}^{t+1}\|^2-\frac{r\left(1-2\beta\right)}{\beta^2}\|\u_{\textrm{ex}}^T-\u_{\textrm{ex}}^{T-1}\|^2  \notag\\
    \leq\ & \frac{1}{2\eta}\|\u^\star -\u^0\|^2-\left(\frac{1}{4\eta}-\frac{L}{2}\right)\|\u^\star -\u^{T}\|^2, \notag 
    \end{align*} 
The last inequality follows from the condition  $0<\eta \leq \tfrac{1}{7(L+r)}$, which ensures $ 1-\beta > 1-2\beta\geq\tfrac{5}{7}>0 $. Rearranging the terms gives 
\[
  \|\u^\star -\u^{T}\|^2\leq  \frac{2}{1-2L\eta} \|\u^\star -\u^0\|^2.
\] 
Using the same condition $0<\eta \leq \tfrac{1}{7(L+r)}$,  we further deduce that for any iterate $T$, 
\[
  \|\u^\star -\u^{T}\|^2\leq   3 \|\u^\star -\u^0\|^2, 
\] 
which completes the proof.
\end{proof}

Armed with Lemma \ref{lem:fund_cc} and Proposition \ref{proposition:cc}, we are ready to establish the optimal iteration complexity of DS-OGDA.  
\begin{proof}[Proof of Theorem \ref{theorem:cc}]
To establish the convergence rate stated in Theorem~\ref{theorem:cc}, we consider the alternative stationarity measure $\tfrac{1}{T} \sum_{t=0}^{T-1} \langle G_{\u}^{t+1}, \u^{t+1} - \u \rangle$, as guaranteed by Lemma~\ref{lem:cc}.  Our goal is to bound this quantity using the upper bound in Proposition~\ref{proposition:cc}(i). Most terms in the bound telescope naturally, leaving only the final inner product term  $\langle G^{t+1}_{\u_{\textrm{ex}}}-G^{t}_{\u_{\textrm{ex}}},  \u_{\textrm{ex}}^{t+1}-\u_{\textrm{ex}}\rangle$, which we denote by term~\O4. Using the update rule of $\u_{\textrm{ex}}^{t+1}$ and the fact that $\beta \leq 1$, we have 
\begin{align}\label{eq:bd_9}
    \text{\O4}
   \leq \ &  \langle G^{t+1}_{\u_{\textrm{ex}}},\u_{\textrm{ex}}^{t+1}-\u_{\textrm{ex}}\rangle-\langle G^{t}_{\u_{\textrm{ex}}}, \u_{\textrm{ex}}^{t}-\u_{\textrm{ex}}\rangle+r\|\u_{\textrm{pd}}^t-\u_{\textrm{ex}}^t\|^2.
\end{align}
Summing both sides of Proposition~\ref{proposition:cc}(i) from $t = 0$ to $T - 1$, and applying \eqref{eq:bd_9} and the fact that $\tfrac{1}{2\eta}-\tfrac{L+r}{2}\geq 3\left(L+r\right) $, we have 
\begin{align}\label{eq:cc_ub_4} 
       &\sum_{t=0}^{T-1} \langle G(\u^{t+1}),\u^{t+1}-\u\rangle \notag\\
       \leq \ & \frac{1}{2\eta}\|\u-\u^0\|^2-\frac{1}{2\eta}\|\u-\u^{T}\|^2-\left(\frac{1}{2\eta}-\frac{L+r}{2}\right)\|\u^{T}-\u^{T-1}\|^2 \notag\\
      & +\underbrace{\langle G^{T}_{\u_{\textrm{pd}}}-G^{T-1}_{\u_{\textrm{pd}}},\u_{\textrm{pd}}^{T}-\u_{\textrm{pd}} \rangle+\langle G^{T}_{\u_{\textrm{ex}}},\u_{\textrm{ex}}^{T}-\u_{\textrm{ex}}\rangle}_{\text{\O5}}+ r\sum_{t=0}^{T-1} \|\u_{\textrm{pd}}^t-\u_{\textrm{ex}}^t\|^2  \notag\\
   \leq\ & \frac{1}{2\eta}\|\u-\u^0\|^2-\left(\frac{1}{2\eta}-\frac{L+r}{2}\right)\|\u-\u^{T}\|^2 +\frac{r}{2} \|\u_{\textrm{pd}}^{T}-\u_{\textrm{ex}}^{T}\|^2\notag \\
   & + r\sum_{t=0}^{T-1} \|\u_{\textrm{pd}}^t-\u_{\textrm{ex}}^t\|^2-\left(\frac{1}{2\eta}- \frac{7\left(L+r\right)}{2} \right)\|\u^{T}-\u^{T-1}\|^2    \notag\\ 
   \leq \ &\frac{1}{2\eta}\|\u-\u^0\|^2 + \frac{3r}{2}\sum_{t=0}^{T-1} \|\u_{\textrm{pd}}^t-\u_{\textrm{ex}}^t\|^2,
    \end{align}   
where the second inequality in follows from an upper bound on the term \text{\O5}. Specifically, the first component of \text{\O5} can be bounded using an argument similar to that of \eqref{eq:detailed_bd}, while the second component is controlled via the Cauchy-Schwarz and AM-GM inequalities. In particular, we obtain
\begin{align} \label{eq:bd_5} 
    \text{\O5} 
    \leq\ &  \frac{L+r}{2}\|\u_{\textrm{pd}}^{T}-\u_{\textrm{pd}}\|^2+3\left(L+r\right)\|\u^T-\u^{T-1}\|^2 +r\langle \u_{\textrm{ex}}^{T}-\u_{\textrm{pd}}^{T}, \u_{\textrm{ex}}^{T}-\u_{\textrm{ex}}\rangle \notag\\
    \leq\ & \frac{L+r}{2}\|\u_{\textrm{pd}}^{T}-\u_{\textrm{pd}}\|^2+3\left(L+r\right)\|\u^T-\u^{T-1}\|^2 +\frac{r}{2}\|\u_{\textrm{pd}}^{T}-\u_{\textrm{ex}}^{T}\|^2  +\frac{r}{2}\|\u_{\textrm{ex}}^{T}-\u_{\textrm{ex}}\|^2.
\end{align}  
The last inequality in \eqref{eq:cc_ub_4} is due to the choice of $\eta$, which ensures that $\tfrac{1}{2\eta}-\tfrac{L+r}{2}>\tfrac{1}{2\eta}- \tfrac{7\left(L+r\right)}{2}  \geq0$.
Due to  Proposition \ref{proposition:cc}(ii), we know, for any $\u_1,\u_2\in \D$, we have
\begin{align*}
     \|\u_1-\u_2\|^2\leq 2\|\u_1-\u^\star\|^2+2\|\u_2-\u^\star\|^2\leq \diam(\D)^2.
\end{align*}
Thus, applying Lemma \ref{lem:fund_cc}, inequality \eqref{eq:cc_ub_4}, and choosing $r\leq \mO(T^{-1})$, we obtain
\begin{align}\label{eq:cc_iter_1}
     \min_{t\in\{0,1,\ldots,T-1\}}\left\{\max_{\y\in \Y} f(\x^t,\y)-\min_{\x\in \X} f(\x,\y^t)\right\}
    \leq\  & \frac{1}{T}\left(\frac{1}{2\eta}\diam(\D)^2+ 3rT\diam(\D)^2\right) \leq \mO(T^{-1}),
\end{align}  
which finishes the proof.  
\end{proof}

\subsection{Proof for Nonconvex Minimax Problems}\label{sec:proof_nc}
In this subsection, we focus on nonconvex minimax problems, including NC-C, C-NC, and NC-NC problems satisfying a one-sided K\L{} condition (see Assumption~\ref{ass:dual}(ii)). Our convergence analysis follows the general framework developed in~\citet{zhang2020single,li2022nonsmooth}, which employs a Lyapunov function to capture the algorithm’s descent behavior.
The main difficulty lies in the absence of primal-dual dominance, which introduces a negative term in the descent estimate and makes it unclear whether the Lyapunov function decreases monotonically. To address this issue, various primal-dual error bounds  have been developed in the literature to balance the primal and dual progress~\citep{zhang2020single,li2022nonsmooth,zheng2023universal}. However, applying this framework to our algorithm presents new challenges due to its specific structure.
To overcome these difficulties, we design a new Lyapunov function tailored to our method, and carefully analyze its descent behavior under appropriate step sizes selection. The primal-dual error bound—originally proposed in \citet{zheng2023universal}—can be also adapted to our setting to ensure sufficient descent.

We construct a new Lyapunov function $\Phi:(\R^{n} \times \R^{d}\times \R^{n} \times \R^{d})^2 \rightarrow \R$ as follows:
 \begin{align}\label{eq:lya_fun}
      & \Phi(\x,\y,\z,\v;\x',\y',\z',\v') \notag\\
       \coloneqq\ & \underbrace{F(\x,\y,\z,\v)-d(\y,\z,\v)}_\text{Primal descent}+\underbrace{p(\z,\v)-d(\y,\z,\v)}_\text{Dual ascent}+\underbrace{q(\z)-p(\z,\v)}_\text{Proximal ascent}   +\underbrace{q(\z)-\underline{F}}_\text{Proximal descent} +\underline{F}\notag\\
       &+\frac{1}{2\eta_{\x}}\|\x-\x'\|^2+\frac{1}{2\eta_{\y}}\|\y-\y'\|^2   +\frac{r_{\x}}{2\beta_{\x}}\|\z-\z'\|^2+\frac{r_{\y}}{2\beta_{\y}}\|\v-\v'\|^2 \notag \\
       =\ & F(\x,\y,\z,\v)-2d(\y,\z,\v)+2q(\z)+\frac{1}{2\eta_{\x}}\|\x-\x'\|^2+\frac{1}{2\eta_{\y}}\|\y-\y'\|^2 \notag\\
       &+\frac{r_{\x}}{2\beta_{\x}}\|\z-\z'\|^2+\frac{r_{\y}}{2\beta_{\y}}\|\v-\v'\|^2, 
    \end{align}
  where $d(\y,\z,\v)\coloneqq\min_{\x\in \X}F(\x,\y,\z,\v)$, $p(\z,\v)\coloneqq\max_{\y\in \Y} d(\y,\z,\v)$, $q(\z)  \coloneqq  \max_{\v\in\R^d} p(\z,\v)$, and $\underline{F}\coloneqq\min_{\z\in \R^n} q(\z)$. 
 Here, four additional quadratic terms are introduced to account for the extrapolation steps. 

 To facilitate the analysis of the Lyapunov function, we introduce several auxiliary  mappings that correspond to the solutions of inner minimization and maximization problems. These mappings capture the implicit updates involved in both smoothing and extrapolation steps and are used repeatedly throughout our descent analysis.
 \begin{definition} We define the following auxiliary optimal solution mappings:
\begin{enumerate}[label=(\roman*)] \setlength{\itemsep}{5pt}
\item $\x(\y,\z,\v)\coloneqq\argmin_{\x\in\X} F(\x,\y,\z,\v)$, 
\item $\y(\x,\z,\v)\coloneqq\argmax_{\y\in \Y} F(\x,\y,\z,\v)$, 
\item $\x(\z,\v)\coloneqq\argmin_{\x\in \X} \max_{\y\in \Y} F(\x,\y,\z,\v)$, 
\item $\v(\z)\coloneqq\argmax_{\v\in \R^d} P(\z,\v)$. 
\end{enumerate}
  \end{definition}
 
 We are now ready to present the basic descent estimate for the constructed Lyapunov function. To ensure the descent property of the Lyapunov function established in Proposition~\ref{proposition:suff_dec}, we impose a set of step size conditions tailored to  our algorithm. These conditions are summarized below.

 \begin{condition}[Step size conditions for basic descent estimate]\label{cond:step} Let $r_{\x}\geq 2L$. Define the constants $\sigma \coloneqq\tfrac{3\eta_{\x} r_{\x}+1}{\eta_{\x}(r_{\x}-L)}$ and $L_d\coloneqq\left(\tfrac{L}{r_{\x}-L}+2\right)  L+r_{\y}$. The step sizes are chosen to satisfy:
 \begin{align*}
  0&<\eta_{\x}\leq \min\left\{\frac{1}{6(4L+r_{\x}+r_{\y})}, \frac{1}{320r_{\y}^2\eta_{\y} }\right\}, \\
  0&<\eta_{\y} \leq  \min\left\{\frac{1}{6(3L+r_{\x}+r_{\y})}, \frac{1}{6  L\sigma^2}, \frac{1}{3(2L_d+15 L+r_{\y})}, \frac{1}{20 L(5L+ 2 )}\right\}, \\
  0 &< \beta_{\x} \leq \min\left\{\frac{2r_{\x}}{80r_{\x} +6r_{\x}^2 +3L^2+12 r_{\x}L},\frac{\eta_{\y}   L^2}{7680 r_{\x}} \right\}, 0  <\beta_{\y} \leq  \frac{\eta_{\y} L^2}{240r_{\y}}.  
  \end{align*}
 \end{condition}

  \begin{proposition}[Basic descent estimate]\label{proposition:suff_dec}
Let $r_{\x}\geq 2L$, $r_{\y}\geq 2 L$, and  step sizes satisfy Condition  \ref{cond:step}. 
  Then, for any $t\geq 0$, we have
   \vspace{-1mm}
   \begin{align}\label{eq:suff_dec}  
	         \Phi^t-\Phi^{t+1} \geq\ &\frac{1}{4\eta_{\x}}  \|\x^t-\x^{t+1}\|^2+\frac{1}{60\eta_{\y}}\|\y^t-\y_+^t(\z^t,\v^t)\|^2+ \frac{r_{\x}}{8\beta_{\x}} \|\z^t-\z^{t+1}\|^2\notag\\
    &+  \frac{r_{\y}}{16\beta_{\y}}\|\v^t-\v_+^t(\z^{t+1})\|^2  + \frac{1}{8\eta_{\x}}\|\x^t-\x^{t-1}\|^2+ \frac{1}{8\eta_{\y}}\|\y^t-\y^{t-1}\|^2\notag\\
    &+ \frac{r_{\x}}{8\beta_{\x}} \|\z^t-\z^{t-1}\|^2 +  \frac{r_{\y}}{8\beta_{\y}} \|\v^t-\v^{t-1}\|^2\notag\\
    & -8r_{\x}\beta_{\x}
          \|\x(\z^{t+1},\v(\z^{t+1}))-\x(\z^{t+1},\v_+^t(\z^{t+1}))\|^2,
	    \end{align} 
 where $\y_{+}(\z,\v)\coloneqq\proj_{\Y}(\y+\eta_{\y} \nabla_{\y}F(\x(\y,\z,\v),\y,\z,\v))$ and 
$\v_{+}(\z)\coloneqq \v  +\beta_{\y}(\y(\x(\z,\v),\z,\v)-\v)$. 
\end{proposition}

To control the negative term in \eqref{eq:suff_dec}, we introduce a primal-dual error bound, which has been comprehensively studied in our prior work \citep[Propositions 2]{zheng2023universal}.  For completeness, we state the error bound result below.

	\begin{proposition}
   \label{proposition:NC_KL}  For any $\z \in \R^n$ and $\v \in \R^d$, we have 
 \begin{enumerate}[label=(\roman*)] \setlength{\itemsep}{0pt}
  \item  Suppose that Assumption \ref {ass:dual}(i) holds and $\Y$ is compact. We have 
   \vspace{-2mm}
   \[
    \|\x(\z^{t+1},\v_+^t(\z^{t+1}))-\x(\z^{t+1},\v(\z^{t+1}))\|^2 \leq \omega_0\|\v_+^t(\z^{t+1})-\v^t\|,
\]
\item  Suppose that Assumption \ref {ass:dual}(ii) holds. We have
 \vspace{-2mm}
 \[
		 \|\x(\z^{t+1},\v_+^t(\z^{t+1}))-\x(\z^{t+1},\v(\z^{t+1}))\|^2\leq \omega_1\|\v_+^t(\z^{t+1})-\v^t\|^{\frac{1}{\theta}},
		\] 
\end{enumerate}
where $\omega_0 \coloneqq\tfrac{4r_{\y}\diam(\Y)}{r_{\x}-L}\left(\tfrac{1-\beta_{\y}}{\beta_{\y}}+\tfrac{r_{\y}}{r_{\y}- L}\right)$ and 
  $\omega_1 \coloneqq\tfrac{2}{(r_{\x}-L)\tau_{\y}}\left(\tfrac{r_{\y}(1-\beta_{\y})}{\beta_{\y}}+\tfrac{r_{\y}^2}{r_{\y}- L}\right)^{\tfrac{1}{\theta}}$. 
	\end{proposition}
Armed with Propositions~\ref{proposition:suff_dec} and~\ref{proposition:NC_KL}, the convergence of DS-OGDA—i.e., Theorem~\ref{theorem:2}—can be established by following the analysis framework in~\citet{zhang2020single, li2022nonsmooth, zheng2023universal}.
Furthermore, the symmetric nature of DS-OGDA allows the descent analysis in Proposition~\ref{proposition:suff_dec} to be extended to C-NC and KŁ-NC problems.
In parallel, an analogous primal-dual error bound under Assumption~\ref{ass:primal} has been established in~\citet[Proposition 6]{zheng2023universal}, which, together with the descent analysis, leads to the proof of Corollary~\ref{col:1}.

\subsection{Proof of Universal Results}\label{sec:proof_universal}
In the previous subsections, we have shown that DS-OGDA achieves convergence across all problem classes. To establish its universal applicability, it remains to verify the feasibility of using symmetric step sizes—i.e., selecting parameters in Condition~\ref{cond:step} such that $\eta_{\x} = \eta_{\y}$ and $\beta_{\x} = \beta_{\y}$. Lemma~\ref{lemma:symm_steps} confirms that such a symmetric parameter set exists. Equipped with this result, we proceed to refine the convergence analysis under the C-C setting to derive an improved iteration complexity, thus completing the proof of Theorem~\ref{theorem:universal}. 
\begin{lemma}[Feasibility of symmetric step sizes] \label{lemma:symm_steps}
Suppose either Assumption \ref {ass:dual}(i), \ref{ass:dual}(ii), \ref{ass:primal}(i) or \ref{ass:primal}(ii) holds. For any $\underline{\varepsilon}>0$ and any integer $T>1$, assume the K\L{} exponent $\theta$ of the primal or dual problem satisfies $\theta \in [\underline{\varepsilon}, 1]$. We consider a symmetric parameter setting: $r := r_{\x} = r_{\y} \geq 74L$, $\eta := \eta_{\x} = \eta_{\y}$, and $\beta := \beta_{\x} = \beta_{\y}$. 
Then, the step sizes satisfy the general requirements in Condition~\ref{cond:step} provided the following hold:
    \[
  \begin{aligned} 
       &  \,\frac{-(L-r)\sqrt{L^2-74Lr+r^2}+L^2-38Lr+r^2}{108Lr^2} \! \leq  \eta \! \leq \min\left\{\!\frac{1}{8\sqrt{5} r}, \frac{1}{20L(5L+2)}\!\right\},\\
       &  0< \beta \leq \min\left \{\frac{2r}{80r+6r^2+3L^2+12rL},\frac{\eta L^2 }{7680r }, \right.\\
       &\quad \quad \quad \quad \quad \quad \left. \frac{\sqrt{(r-L)\tau}\left(\beta(r-L)\right)^{\frac{1}{2\underline{\varepsilon}}}}{8\sqrt{ 2\left(2r\left(r-L+r\beta\right)\right) ^{\frac{1}{\underline{\varepsilon}}}\max\left\{\diam(\Y),\diam(\X)\right\}^{\frac{1}{\underline{\varepsilon}}-2} }}  \right\}\leq \mO(T^{-1/2}).
  \end{aligned}
  \]
\end{lemma}

\begin{proof} 
If $r\ge 2L$,  Condition~\ref{cond:step} becomes 
 \begin{subequations}
     \begin{equation}\label{eq:bd1}
     \frac{6L(3 \eta r+1)^2}{ (r-L)^2}\leq \eta \leq \min\left\{\frac{1}{12(2L+r)}, \frac{1}{3(2L_d+15L+r)} , \frac{1}{20L(5L+2)},\frac{1}{8\sqrt{5} r} \right \},
  \end{equation}
 \begin{equation}\label{eq:bd2}
 \begin{aligned} 
      0< \beta \leq\  \min &\left \{\frac{2r}{80r+6r^2+3L^2+12rL},\frac{\eta L^2 }{7680r }, \right.\\
        & \ \ \left. \frac{\sqrt{(r-L)\tau}\left(\beta(r-L)\right)^{\frac{1}{2\underline{\varepsilon}}}}{8\sqrt{ 2\left(2r\left(r-L+r\beta\right)\right) ^{\frac{1}{\underline{\varepsilon}}}\max\left\{\diam(\Y),\diam(\X)\right\}^{\frac{1}{\underline{\varepsilon}}-2} }}  \right\}\leq \mO(T^{-1/2}). 
 \end{aligned}
 \end{equation}
 \end{subequations}

 Next, we focus on the inequality in \eqref{eq:bd1} and solve the resulting quadratic inequality in $\eta$. This yields the following valid interval:
\begin{align}\label{eq:universal_nc}
    \eta\geq\ & \frac{-(r-L)\sqrt{L^2-74Lr+r^2}+L^2-38Lr+r^2}{108Lr^2}, \notag\\
    \eta \leq\ & \frac{(r-L)\sqrt{L^2-74Lr+r^2}+L^2-38Lr+r^2}{108Lr^2}.
\end{align}
To ensure the positivity of the lower bound, we impose the constraint  $r\geq 74L$. Substituting this into \eqref{eq:universal_nc} and combining with the upper bound in \eqref{eq:bd1} yields
\begin{align*}
    &\frac{-(r-L)\sqrt{L^2-74Lr+r^2}+L^2-38Lr+r^2}{108Lr^2} \leq  \eta  \leq \min\left\{\frac{1}{8\sqrt{5}r}, \frac{1}{20L(5L+2)}\right\}.
\end{align*}
 Therefore, we finish the proof.
\end{proof}
With the symmetric step-size choices established in Lemma~\ref{lemma:symm_steps}, the universal convergence result in Theorem~\ref{theorem:universal}(i) follows directly from Theorem~\ref{theorem:2} and Corollary~\ref{col:1}. What remains is to refine the analysis for the C-C case to derive the improved iteration complexity under the same step-size conditions.

\begin{proof}[Proof of Theorem \ref{theorem:universal}(ii)]
   We start to derive a general upper bound for  $\langle G_{\u}^{t+1},\u^{t+1}-\u\rangle$. To this end, we remove the condition $\beta = \eta r$  previously used in the derivation of \eqref{eq:cc_0}, combine this with the bound in \eqref{eq:detailed_bd}, which yields: 
    \begin{align} 
        &\langle G^{t+1}_{\u_{\textrm{pd}}},\u_{\textrm{pd}}^{t+1}-\u_{\textrm{pd}}\rangle \notag\\
       \leq\ & \frac{1}{2\eta} \|\u_{\textrm{pd}}-\u_{\textrm{pd}}^t\|^2\ -\frac{1}{2\eta}\|\u_{\textrm{pd}}-\u_{\textrm{pd}}^{t+1}\|^2 -\left(\frac{1}{2\eta}-\frac{L+r}{2}\right)\|\u_{\textrm{pd}}^{t+1}-\u_{\textrm{pd}}^t\|^2\notag \\
      & + 3\left(L+r\right)\|\u^t-\u^{t-1}\|^2 +\langle G^{t+1}_{\u_{\textrm{pd}}}-G_{\u_{\textrm{pd}}}^t,\u_{\textrm{pd}}^{t+1}-\u_{\textrm{pd}}\rangle - \langle G^{t}_{\u_{\textrm{pd}}}-G_{\u_{\textrm{pd}}}^{t-1},\u_{\textrm{pd}}^{t}-\u_{\textrm{pd}}\rangle, 
        \label{eq:compact_1}   \, \text{and}\\
     &\langle G^{t+1}_{\u_{\textrm{ex}}},\u_{\textrm{ex}}^{t+1}-\u_{\textrm{ex}}\rangle \notag\\
    \leq\  &\frac{r}{2\beta} \|\u_{\textrm{ex}}-\u_{\textrm{ex}}^t\|^2\ -\frac{r}{2\beta}\|\u_{\textrm{ex}}-\u_{\textrm{ex}}^{t+1}\|^2-\frac{r}{2\beta}\|\u_{\textrm{ex}}^{t+1}-\u_{\textrm{ex}}^t\|^2 + \langle G^{t+1}_{\u_{\textrm{ex}}},  \u_{\textrm{ex}}^{t+1}-\u_{\textrm{ex}}\rangle \notag\\
    & - \langle G_{\u_{\textrm{ex}}}^t,  \u_{\textrm{ex}}^{t}-\u_{\textrm{ex}}\rangle    +\frac{r}{\beta}\|\u_{\textrm{ex}}^{t+1}-\u_{\textrm{ex}}^t\|^2  \label{eq:compact_4}. 
\end{align}
Summing \eqref{eq:compact_1} and \eqref{eq:compact_4} from $t = 0$ to $T - 1$ yields, 
\begin{align}
    \label{eq:compact_main}
    &\sum_{t=0}^{T-1} \langle G_{\u}^{t+1},\u^{t+1}-\u\rangle   \notag \\
        \leq\ & 
        \frac{1}{2\eta} \|\u_{\textrm{pd}}-\u_{\textrm{pd}}^0\|^2\ -\frac{1}{2\eta}\|\u_{\textrm{pd}}-\u_{\textrm{pd}}^{T}\|^2+\frac{r}{2\beta} \|\u_{\textrm{ex}}-\u_{\textrm{ex}}^0\|^2\ -\frac{r}{2\beta}\|\u_{\textrm{ex}}-\u_{\textrm{ex}}^{T}\|^2 \notag \\
        &-3\left(L+r\right)\|\u^T-\u^{T-1}\|^2
      +\langle G^{T}_{\u_{\textrm{pd}}}- G^{T-1}_{\u_{\textrm{pd}}},\u_{\textrm{pd}}^{T}-\u_{\textrm{pd}}\rangle   +\langle G^{T}_{\u_{\textrm{ex}}},  \u_{\textrm{ex}}^{T}-\u_{\textrm{ex}}\rangle  +\frac{r}{\beta} \sum_{t=0}^{T-1}  \|\u_{\textrm{ex}}^{t+1}-\u_{\textrm{ex}}^t\|^2 \notag \\
   \leq\ &  \frac{1}{2\eta} \|\u_{\textrm{pd}}-\u_{\textrm{pd}}^0\|^2\ -\left(\frac{1}{2\eta}-\frac{L+r}{2}\right)\|\u_{\textrm{pd}}-\u_{\textrm{pd}}^{T}\|^2+\frac{r}{2\beta} \|\u_{\textrm{ex}}-\u_{\textrm{ex}}^0\|^2\  -\frac{r\left(1-\beta\right)}{2\beta}\|\u_{\textrm{ex}} -\u_{\textrm{ex}}^{T}\|^2 \notag \\
   & + \frac{r}{2} \|\u_{\textrm{pd}}^T-\u_{\textrm{ex}}^T\|^2  + \frac{r}{\beta} \sum_{t=0}^{T-1} \|\u_{\textrm{ex}}^{t+1}-\u_{\textrm{ex}}^t\|^2  \notag \\
   \leq\ &  \frac{1}{2\eta} \|\u_{\textrm{pd}}-\u_{\textrm{pd}}^0\|^2 +\frac{r}{2\beta} \|\u_{\textrm{ex}}-\u_{\textrm{ex}}^0\|^2\   + \frac{r}{2} \|\u_{\textrm{pd}}^T-\u_{\textrm{ex}}^T\|^2  + \frac{r}{\beta} \sum_{t=0}^{T-1}  \|\u_{\textrm{ex}}^{t+1}-\u_{\textrm{ex}}^t\|^2,
\end{align}  
where the first inequality holds since $\tfrac{1}{2\eta}-\tfrac{L+r}{2} \geq 3\left(L + r\right)$  and $\tfrac{r}{2\beta} \geq 3\left(L + r\right)$, the second one follows from~\eqref{eq:bd_5}.
Moreover, from Proposition~\ref{proposition:suff_dec}, we have
\begin{align*} 
    \sum_{t=0}^T  \frac{\|\u_{\textrm{ex}}^{t+1}-\u_{\textrm{ex}}^t\|^2}{\beta} \leq\ & \Phi^0-\Phi^T +8r\beta \sum_{t=0}^T \|\x(\z^{t+1},\v(\z^{t+1}))-\x(\z^{t+1},\v_+^t(\z^{t+1}))\|^2\\
    \leq\ & \Phi^0+8Tr\beta \diam(\X)^2 
    \leq\mO(T^{1/2}). 
\end{align*} The symmetric nature of DS-OGDA ensures that the bound $ \sum_{t=0}^T  \tfrac{\|\u_{\textrm{ex}}^{t+1}-\u_{\textrm{ex}}^t\|^2}{\beta} \leq \mO(T^{1/2})$ also holds in the C-NC and K\L{}-NC problem settings.
Incorporating this with \eqref{eq:compact_main} yields
\[
\begin{aligned}
     \frac{1}{T}\sum_{t=0}^{T-1} \langle G_{\u}^{t+1},\u^{t+1}-\u\rangle\leq \mO(T^{-1/2}).
\end{aligned}
\]
The proof is completed by invoking Lemma~\ref{lem:fund_cc}.
\end{proof} 
 
\section{Suboptimal Performance of Existing Algorithms} \label{sec:tight} 
In this section, we review several widely used methods in the literature and examine why they fall short of being universal approaches for minimax problems. Their primary limitation lies in the failure to meet the second criterion of optimal achievability—specifically, these algorithms are unable to attain optimal or best-known performance guarantees even when the underlying structural assumptions are known and verified. 

As mentioned earlier, there are two primary algorithmic frameworks for designing single-loop methods for minimax problems. The first category comprises VI-based methods, with EG and OGDA being the most prominent. These methods achieve optimal iteration complexity for C-C problems. However, in the NC-C setting, their iteration complexity degrades to $\mO(\epsilon^{-6})$, which is worse than the best-known results. Moreover, this slower rate has been shown to be tight in \citet{mahdavinia2022tight}, underscoring a key limitation: These methods fail to satisfy the optimal achievability criterion. The second category includes algorithms in which the primal and dual variables are not updated simultaneously, such as S-GDA \citep{zhang2020single} and DS-GDA \citep{zheng2023universal}. These methods are tailored to one-sided dominance conditions and are known to attain the best-known complexities for various nonconvex minimax problems. However, it remains unclear whether they can achieve the optimal rate in the C-C setting.

In the following, we show that even when the constraint sets are compact, DS-GDA—i.e., Algorithm~\ref{alg:1} with the red part removed—achieves only a suboptimal iteration complexity of $\mO(\epsilon^{-2})$ in the convex-concave (C-C) setting. Furthermore, we demonstrate that this suboptimality is not an artifact of analysis—it is tight.

\begin{theorem}[Convergence result of DS-GDA for C-C problems]\label{theorem:dsgda}
 Suppose that Assumptions \ref{ass:dual} (i) and \ref{ass:primal} (i) hold and both $\X$ and $\Y$ are compact. For any integer $T \geq 0$, if we set $0 < \eta \coloneq \eta_{\x} = \eta_{\y} \leq \mO(T^{-1/2})$, $0  \leq r \coloneq  r_{\x} = r_{\y} \leq \mO(1)$, and $0 <  \beta  \coloneq\beta_{\x} = \beta_{\y}=  \eta r \leq \mO(T^{-1/2})$, then there exists an index $t\in\{0,1,\ldots, T-1\}$ such that
 $\left(\x^{t+1},  \y^{t+1} \right)$ is an $\mO(T^{-1/2})$-SP of problem \eqref{eq:prob}.
\end{theorem} 
 \begin{proof}
 By  applying Lemma \ref{lem:ippm} with $\meps^t =\eta \left(G^{t+1}_{\u}-G^{t}_{\u}\right)$, we have
\begin{align}\label{eq:ds_gda}
    \langle G_{\u}^{t+1} ,\u^{t+1}-\u\rangle \leq\ & \frac{1}{2\eta}\|\u-\u^{t}\|^2- \frac{1}{2\eta}\|\u-\u^{t+1}\|^2- \frac{1}{2\eta} \|\u^t-\u^{t+1}\|^2  +\langle G^{t+1}_{\u}-G^{t}_{\u},\u^{t+1}-\u\rangle.
\end{align} 
For the inner product term, we decompose and bound it as follows:
\begin{align}\label{eq:bd_dsgda}
     &\langle G^{t+1}_{\u}-G^{t}_{\u},\u^{t+1}-\u\rangle \notag\\
     =\ & \langle G^{t+1}_{\u_{\textrm{pd}}}-G^{t}_{\u_{\textrm{pd}}},\u_{\textrm{pd}}^{t+1}-\u_{\textrm{pd}}\rangle  +  \langle G^{t+1}_{\u_{\textrm{ex}}}-G^{t}_{\u_{\textrm{ex}}},\u_{\textrm{ex}}^{t+1}-\u_{\textrm{ex}}\rangle \notag\\
\leq \ &  \left(L+r\right) \left(\|\x^t-\x^{t+1}\|+\|\y^t-\y^{t+1}\|+\|\z^t-\z^{t+1}\|\right)\|\x^{t+1}-\x\|  \notag\\
&+ \left(L+r\right) \left(\|\x^t-\x^{t+1}\|+\|\y^t-\y^{t+1}\|+\|\v^t-\v^{t+1}\|\right)\|\y^{t+1}-\y\| \notag\\
&+r \left(\|\x^t-\x^{t+1}\|+\|\z^t-\z^{t+1}\|\right)\|\z^{t+1}-\z\|  \notag\\
&+r \left( \|\y^t-\y^{t+1}\|+\|\v^t-\v^{t+1}\|\right) \|\v^{t+1}-\v\| \notag\\
\leq\ &  \left(L+r\right) \left(\|\x^t-\x^{t+1}\|+\|\y^t-\y^{t+1}\|+\|\z^t-\z^{t+1}\|+\|\v^{t+1}-\v^t\|\right) \notag\\
&\times \left(\|\x^{t+1}-\x\|+\|\y^{t+1}-\y\|+\|\z^{t+1}-\z\|+\|\v^{t+1}-\v\|\right) \notag\\ 
\leq\ & 4\eta \left(L+r\right)^2 \|\u-\u^{t+1}\|^2+\frac{1}{4\eta}\|\u^{t+1}-\u^t\|^2,
\end{align}
where the inequality follows from \eqref{eq:bd_cs} and the Cauchy-Schwarz inequality. The last inequality is derived from the AM–GM inequality.
Plugging \eqref{eq:bd_dsgda} into \eqref{eq:ds_gda} and taking summation over two sides from $t=0$ to $T-1$ yields
\[
\begin{aligned}
    \frac{1}{T}\sum_{t=0}^{T-1} \langle G_{\u}^{t+1} ,\u^{t+1}-\u\rangle 
    \leq\ & \frac{1}{2\eta T}\|\u-\u^0\|^2+\frac{1}{T}\sum_{t=0}^{T-1} 4\eta \left(L+r\right)^2 \|\u-\u^{t+1}\|^2 
    \leq\  \mO(T^{-1/2}),
\end{aligned}
\]
where the last inequality is due to $\eta\leq \mO(T^{-1/2})$ and the boundedness of constraint sets.The proof is completed by invoking Lemma~\ref{lem:fund_cc}.
\end{proof}

When $\beta = \eta r$, the pairs $(\x, \z)$ and $(\y, \v)$ can be treated jointly as single primal and dual variables, respectively. Owing to the joint convexity-concavity of the function $F$ (see Lemma~\ref{lem:cc}), DS-GDA can be interpreted as a vanilla GDA method applied directly to $F$. To establish the tightness of the analysis in Theorem~\ref{theorem:dsgda}, it therefore suffices to show that the iteration complexity of vanilla GDA is tight, as demonstrated below.

\begin{theorem}[Tightness of GDA]\label{theorem:tight}
For any $T\ge 1$, let  $\{(\x^t,\y^t)\}_{t=0}^T$ generated by GDA with a step-size $\eta=\mO(T^{-1/2})$. Then, there exists a convex-concave smooth minimax problem instance such that, in order for some iterate $(\x^t,\y^t)$, with $t\leq T$, to be an $\epsilon$-saddle point of problem \eqref{eq:prob}, it is necessary that $T=\Omega(\epsilon^{-2})$. 
\end{theorem} 
 
\begin{proof} 
Inspired by \citet{mahdavinia2022tight}, we consider the following hard instance 
    \[
        \min_{x\in \R } \max_{y\in [0,1]} f(x,y):= \frac{1}{2}x^2y.
    \] 
We first claim that if we choose $|x^0| \leq 1 $, then for any $t \geq 0$, it holds that $ |x^t| \leq 1 $. We prove this by induction.

\textbf{Base case:} When $t = 0$, the claim holds trivially by assumption.

\textbf{Inductive step:} Assume that $|x^k| \leq 1 $ for some $k \geq 0$. We now show that $|x^{k+1}| \leq 1$. To this end, observe that 
\[
x^{k+1}= x^k-\eta x^k y^k =\left(1-\eta y^k\right) x^k.
\]
Since $0\leq y^k\leq 1$, we have $0\leq 1-\eta y^k\leq 1$, which implies $|x^{k+1}|\leq 1$. 
Therefore, we conclude our claim is right. Next, we can bound $|x^t|$ as follows:
\[
|x^t|=\left|\prod_{k=1}^t \left(1-\eta y^k\right) x^0\right|\geq \left(1-\eta\right)^t|x^0|.
\]
Then, if we choose $x^0=\sqrt{2\epsilon}$, we have 
\begin{align*}
     \epsilon\geq \max_{y' \in [0,1]}f(x^t,y')-\min_{x' \in\R} f(x',y^t)&= \frac{(x^t)^2}{2} \geq \left(1-\eta\right)^{2t}\epsilon, \quad \forall t \leq T, 
\end{align*} 
which implies  
\[
T=\Omega\left(\frac{1}{\eta}\right)=\Omega \left(\frac{1}{\epsilon^2}\right).
\]
We finished our proof. 
\end{proof}

\section{Conclusion}\label{sec:conclusion}
In this paper, we propose a \emph{Doubly Smoothed Optimistic Gradient Descent Ascent} (DS-OGDA) method that achieves both \emph{universal applicability} and \emph{optimal achievability} for smooth minimax problems. With a single set of parameters, DS-OGDA attains an iteration complexity of $\mathcal{O}(\epsilon^{-2})$ for C-C problems and $\mathcal{O}(\epsilon^{-4})$ for general nonconvex minimax problems. Furthermore, when additional structural information is available, DS-OGDA can match the optimal complexity of $\mathcal{O}(\epsilon^{-1})$ in the C-C setting. For nonconvex problems, if a one-sided K\L{} exponent $\theta$ is known, DS-OGDA achieves an $\epsilon$-game stationary point with iteration complexity $\mathcal{O}(\epsilon^{-(4\theta - 2)_+ - 2})$, aligning with the best-known results.
We also demonstrate the necessity of this new algorithm by analyzing the suboptimal rate of DS-GDA in the C-C setting, where its upper-bound complexity has already been shown to be tight.

\section*{Acknowledgment}
Jiajin Li was supported by a Natural Sciences and Engineering Research Council of Canada Discovery Grant GR034865.
 
\bibliographystyle{abbrvnat}
\bibliography{references}

\begin{thebibliography}{46}
\providecommand{\natexlab}[1]{#1}
\providecommand{\url}[1]{\texttt{#1}}
\expandafter\ifx\csname urlstyle\endcsname\relax
  \providecommand{\doi}[1]{doi: #1}\else
  \providecommand{\doi}{doi: \begingroup \urlstyle{rm}\Url}\fi

\bibitem[Arjovsky et~al.(2017)Arjovsky, Chintala, and Bottou]{arjovsky2017wasserstein}
M.~Arjovsky, S.~Chintala, and L.~Bottou.
\newblock Wasserstein generative adversarial networks.
\newblock In \emph{Proceedings of the 34th International Conference on Machine Learning (ICML 2017)}, pages 214--223. PMLR, 2017.

\bibitem[Attouch et~al.(2010)Attouch, Bolte, Redont, and Soubeyran]{attouch2010proximal}
H.~Attouch, J.~Bolte, P.~Redont, and A.~Soubeyran.
\newblock Proximal alternating minimization and projection methods for nonconvex problems: An approach based on the {K}urdyka-{{\L}}ojasiewicz inequality.
\newblock \emph{Mathematics of Operations Research}, 35\penalty0 (2):\penalty0 438--457, 2010.

\bibitem[Attouch et~al.(2013)Attouch, Bolte, and Svaiter]{attouch2013convergence}
H.~Attouch, J.~Bolte, and B.~F. Svaiter.
\newblock Convergence of descent methods for semi-algebraic and tame problems: proximal algorithms, forward--backward splitting, and regularized gauss--seidel methods.
\newblock \emph{Mathematical Programming}, 137\penalty0 (1-2):\penalty0 91--129, 2013.

\bibitem[Ben-Tal et~al.(2009)Ben-Tal, El~Ghaoui, and Nemirovski]{ben2009robust}
A.~Ben-Tal, L.~El~Ghaoui, and A.~Nemirovski.
\newblock \emph{Robust Optimization}, volume~28.
\newblock Princeton University Press, 2009.

\bibitem[Bertsimas et~al.(2011)Bertsimas, Brown, and Caramanis]{bertsimas2011theory}
D.~Bertsimas, D.~B. Brown, and C.~Caramanis.
\newblock Theory and applications of robust optimization.
\newblock \emph{SIAM Review}, 53\penalty0 (3):\penalty0 464--501, 2011.

\bibitem[Blanchet et~al.(2024)Blanchet, Li, Lin, and Zhang]{blanchet2024distributionally}
J.~Blanchet, J.~Li, S.~Lin, and X.~Zhang.
\newblock Distributionally robust optimization and robust statistics.
\newblock \emph{arXiv preprint arXiv:2401.14655}, 2024.

\bibitem[B{\"o}hm(2023)]{bohm2023solving}
A.~B{\"o}hm.
\newblock Solving nonconvex-nonconcave min-max problems exhibiting weak minty solutions.
\newblock \emph{Transactions on Machine Learning Research}, 2023.
\newblock ISSN 2835-8856.

\bibitem[Bolte et~al.(2007)Bolte, Daniilidis, and Lewis]{bolte2007lojasiewicz}
J.~Bolte, A.~Daniilidis, and A.~Lewis.
\newblock The {{\L}}ojasiewicz inequality for nonsmooth subanalytic functions with applications to subgradient dynamical systems.
\newblock \emph{SIAM Journal on Optimization}, 17\penalty0 (4):\penalty0 1205--1223, 2007.

\bibitem[Cai and Zheng(2022)]{cai2022accelerated}
Y.~Cai and W.~Zheng.
\newblock Accelerated single-call methods for constrained min-max optimization.
\newblock In \emph{OPT 2022: Optimization for Machine Learning (NeurIPS 2022 Workshop)}, 2022.

\bibitem[Cai et~al.(2024)Cai, Oikonomou, and Zheng]{cai2024accelerated}
Y.~Cai, A.~Oikonomou, and W.~Zheng.
\newblock Accelerated algorithms for constrained nonconvex-nonconcave min-max optimization and comonotone inclusion.
\newblock In \emph{Proceedings of the 41st International Conference on Machine Learning (ICML 2024)}, pages 5312--5347. PMLR, 2024.

\bibitem[Chen et~al.(2021)Chen, Zhou, Xu, and Liang]{chen2021proximal}
Z.~Chen, Y.~Zhou, T.~Xu, and Y.~Liang.
\newblock Proximal gradient descent-ascent: Variable convergence under k{\l} geometry.
\newblock In \emph{Proceedings of the 9th International Conference on Learning Representations (ICLR 2021)}, 2021.

\bibitem[Dai et~al.(2018)Dai, Shaw, Li, Xiao, He, Liu, Chen, and Song]{dai2018sbeed}
B.~Dai, A.~Shaw, L.~Li, L.~Xiao, N.~He, Z.~Liu, J.~Chen, and L.~Song.
\newblock Sbeed: Convergent reinforcement learning with nonlinear function approximation.
\newblock In \emph{Proceedings of the 35th International Conference on Machine Learning (ICML 2018)}, pages 1125--1134. PMLR, 2018.

\bibitem[Daskalakis et~al.(2018)Daskalakis, Ilyas, Syrgkanis, and Zeng]{daskalakis2017training}
C.~Daskalakis, A.~Ilyas, V.~Syrgkanis, and H.~Zeng.
\newblock Training {GAN}s with optimism.
\newblock In \emph{Proceedings of the 6th International Conference on Learning Representations (ICLR 2018)}, 2018.

\bibitem[Diakonikolas et~al.(2021)Diakonikolas, Daskalakis, and Jordan]{diakonikolas2021efficient}
J.~Diakonikolas, C.~Daskalakis, and M.~Jordan.
\newblock Efficient methods for structured nonconvex-nonconcave min-max optimization.
\newblock In \emph{International Conference on Artificial Intelligence and Statistics}, pages 2746--2754. PMLR, 2021.

\bibitem[Ghadimi et~al.(2019)Ghadimi, Lan, and Zhang]{ghadimi2019generalized}
S.~Ghadimi, G.~Lan, and H.~Zhang.
\newblock Generalized uniformly optimal methods for nonlinear programming.
\newblock \emph{Journal of Scientific Computing}, 79:\penalty0 1854--1881, 2019.

\bibitem[Goodfellow et~al.(2020)Goodfellow, Pouget-Abadie, Mirza, Xu, Warde-Farley, Ozair, Courville, and Bengio]{goodfellow2020generative}
I.~Goodfellow, J.~Pouget-Abadie, M.~Mirza, B.~Xu, D.~Warde-Farley, S.~Ozair, A.~Courville, and Y.~Bengio.
\newblock Generative adversarial networks.
\newblock \emph{Communications of the ACM}, 63\penalty0 (11):\penalty0 139--144, 2020.

\bibitem[Jin et~al.(2020)Jin, Netrapalli, and Jordan]{jin2020local}
C.~Jin, P.~Netrapalli, and M.~Jordan.
\newblock What is local optimality in nonconvex-nonconcave minimax optimization?
\newblock In \emph{Proceedings of the 37th International Conference on Machine Learning (ICML 2020)}, pages 4880--4889. PMLR, 2020.

\bibitem[Kavis et~al.(2019)Kavis, Levy, Bach, and Cevher]{kavis2019unixgrad}
A.~Kavis, K.~Y. Levy, F.~Bach, and V.~Cevher.
\newblock Unixgrad: A universal, adaptive algorithm with optimal guarantees for constrained optimization.
\newblock In \emph{Advances in Neural Information Processing Systems}, volume~32, 2019.

\bibitem[Korpelevich(1976)]{korpelevich1976extragradient}
G.~M. Korpelevich.
\newblock The extragradient method for finding saddle points and other problems.
\newblock \emph{Matecon}, 12:\penalty0 747--756, 1976.

\bibitem[Kuhn et~al.(2019)Kuhn, Esfahani, Nguyen, and Shafieezadeh-Abadeh]{kuhn2019wasserstein}
D.~Kuhn, P.~M. Esfahani, V.~A. Nguyen, and S.~Shafieezadeh-Abadeh.
\newblock Wasserstein distributionally robust optimization: Theory and applications in machine learning.
\newblock In \emph{Operations research \& management science in the age of analytics}, pages 130--166. Informs, 2019.

\bibitem[Lan(2015)]{lan2015bundle}
G.~Lan.
\newblock Bundle-level type methods uniformly optimal for smooth and nonsmooth convex optimization.
\newblock \emph{Mathematical Programming}, 149\penalty0 (1):\penalty0 1--45, 2015.

\bibitem[Levy et~al.(2018)Levy, Yurtsever, and Cevher]{levy2018online}
K.~Y. Levy, A.~Yurtsever, and V.~Cevher.
\newblock Online adaptive methods, universality and acceleration.
\newblock In \emph{Advances in Neural Information Processing Systems}, volume~31, 2018.

\bibitem[Li et~al.(2025)Li, Zhu, and So]{li2022nonsmooth}
J.~Li, L.~Zhu, and A.~M.-C. So.
\newblock Nonsmooth nonconvex--nonconcave minimax optimization: Primal--dual balancing and iteration complexity analysis.
\newblock \emph{Mathematical Programming}, pages 1--51, 2025.

\bibitem[Li and Lan(2023)]{li2023simple}
T.~Li and G.~Lan.
\newblock A simple uniformly optimal method without line search for convex optimization.
\newblock \emph{arXiv preprint arXiv:2310.10082}, 2023.

\bibitem[Lin et~al.(2020{\natexlab{a}})Lin, Jin, and Jordan]{lin2020gradient}
T.~Lin, C.~Jin, and M.~Jordan.
\newblock On gradient descent ascent for nonconvex-concave minimax problems.
\newblock In \emph{Proceedings of the 37th International Conference on Machine Learning (ICML 2020)}, pages 6083--6093. PMLR, 2020{\natexlab{a}}.

\bibitem[Lin et~al.(2020{\natexlab{b}})Lin, Jin, and Jordan]{lin2020near}
T.~Lin, C.~Jin, and M.~I. Jordan.
\newblock Near-optimal algorithms for minimax optimization.
\newblock In \emph{Conference on Learning Theory}, pages 2738--2779. PMLR, 2020{\natexlab{b}}.

\bibitem[Lu(2022)]{lu2022s}
H.~Lu.
\newblock An $\mathcal{O}(s^r)$-resolution ode framework for understanding discrete-time algorithms and applications to the linear convergence of minimax problems.
\newblock \emph{Mathematical Programming}, 194\penalty0 (1):\penalty0 1061--1112, 2022.

\bibitem[Mahdavinia et~al.(2022)Mahdavinia, Deng, Li, and Mahdavi]{mahdavinia2022tight}
P.~Mahdavinia, Y.~Deng, H.~Li, and M.~Mahdavi.
\newblock Tight analysis of extra-gradient and optimistic gradient methods for nonconvex minimax problems.
\newblock In \emph{Advances in Neural Information Processing Systems}, pages 31213--31225, 2022.

\bibitem[Mertikopoulos et~al.(2019)Mertikopoulos, Lecouat, Zenati, Foo, Chandrasekhar, and Piliouras]{mertikopoulos2019optimistic}
P.~Mertikopoulos, B.~Lecouat, H.~Zenati, C.-S. Foo, V.~Chandrasekhar, and G.~Piliouras.
\newblock Optimistic mirror descent in saddle-point problems: Going the extra (gradient) mile.
\newblock In \emph{Proceedings of the 6th International Conference on Learning Representations (ICLR 2018)}, 2019.

\bibitem[Mokhtari et~al.(2020{\natexlab{a}})Mokhtari, Ozdaglar, and Pattathil]{mokhtari2020unified}
A.~Mokhtari, A.~Ozdaglar, and S.~Pattathil.
\newblock A unified analysis of extra-gradient and optimistic gradient methods for saddle point problems: Proximal point approach.
\newblock In \emph{International Conference on Artificial Intelligence and Statistics}, pages 1497--1507. PMLR, 2020{\natexlab{a}}.

\bibitem[Mokhtari et~al.(2020{\natexlab{b}})Mokhtari, Ozdaglar, and Pattathil]{mokhtari2020convergence}
A.~Mokhtari, A.~E. Ozdaglar, and S.~Pattathil.
\newblock Convergence rate of $\mathcal{O}(1/k)$ for optimistic gradient and extragradient methods in smooth convex-concave saddle point problems.
\newblock \emph{SIAM Journal on Optimization}, 30\penalty0 (4):\penalty0 3230--3251, 2020{\natexlab{b}}.

\bibitem[Nemirovski(2004)]{nemirovski2004prox}
A.~Nemirovski.
\newblock Prox-method with rate of convergence $\mathcal{O}(1/t)$ for variational inequalities with lipschitz continuous monotone operators and smooth convex-concave saddle point problems.
\newblock \emph{SIAM Journal on Optimization}, 15\penalty0 (1):\penalty0 229--251, 2004.

\bibitem[Nesterov(2015)]{nesterov2015universal}
Y.~Nesterov.
\newblock Universal gradient methods for convex optimization problems.
\newblock \emph{Mathematical Programming}, 152\penalty0 (1):\penalty0 381--404, 2015.

\bibitem[Nouiehed et~al.(2019)Nouiehed, Sanjabi, Huang, Lee, and Razaviyayn]{nouiehed2019solving}
M.~Nouiehed, M.~Sanjabi, T.~Huang, J.~D. Lee, and M.~Razaviyayn.
\newblock Solving a class of non-convex min-max games using iterative first order methods.
\newblock In \emph{Advances in Neural Information Processing Systems}, pages 14934--14942, 2019.

\bibitem[Omidshafiei et~al.(2017)Omidshafiei, Pazis, Amato, How, and Vian]{omidshafiei2017deep}
S.~Omidshafiei, J.~Pazis, C.~Amato, J.~P. How, and J.~Vian.
\newblock Deep decentralized multi-task multi-agent reinforcement learning under partial observability.
\newblock In \emph{Proceedings of the 34th International Conference on Machine Learning (ICML 2017)}, pages 2681--2690. PMLR, 2017.

\bibitem[Ostrovskii et~al.(2021)Ostrovskii, Lowy, and Razaviyayn]{ostrovskii2021efficient}
D.~M. Ostrovskii, A.~Lowy, and M.~Razaviyayn.
\newblock Efficient search of first-order nash equilibria in nonconvex-concave smooth min-max problems.
\newblock \emph{SIAM Journal on Optimization}, 31\penalty0 (4):\penalty0 2508--2538, 2021.

\bibitem[Pethick et~al.(2022)Pethick, Patrinos, Fercoq, Cevher{\aa}, et~al.]{pethick2022escaping}
T.~Pethick, P.~Patrinos, O.~Fercoq, V.~Cevher{\aa}, et~al.
\newblock Escaping limit cycles: Global convergence for constrained nonconvex-nonconcave minimax problems.
\newblock In \emph{Proceedings of the 10th International Conference on Learning Representations (ICLR 2022)}, 2022.

\bibitem[Rahimian and Mehrotra(2022)]{rahimian2019distributionally}
H.~Rahimian and S.~Mehrotra.
\newblock Frameworks and results in distributionally robust optimization.
\newblock \emph{Open Journal of Mathematical Optimization}, 3:\penalty0 1--85, 2022.

\bibitem[Rakhlin and Sridharan(2013)]{rakhlin2013online}
A.~Rakhlin and K.~Sridharan.
\newblock Online learning with predictable sequences.
\newblock In \emph{Conference on Learning Theory}, pages 993--1019. PMLR, 2013.

\bibitem[Thekumparampil et~al.(2019)Thekumparampil, Jain, Netrapalli, and Oh]{thekumparampil2019efficient}
K.~K. Thekumparampil, P.~Jain, P.~Netrapalli, and S.~Oh.
\newblock Efficient algorithms for smooth minimax optimization.
\newblock In \emph{Advances in Neural Information Processing Systems}, pages 12680--12691, 2019.

\bibitem[Xu et~al.(2023)Xu, Zhang, Xu, and Lan]{xu2023unified}
Z.~Xu, H.~Zhang, Y.~Xu, and G.~Lan.
\newblock A unified single-loop alternating gradient projection algorithm for nonconvex--concave and convex--nonconcave minimax problems.
\newblock \emph{Mathematical Programming}, 201\penalty0 (1):\penalty0 635--706, 2023.

\bibitem[Yang et~al.(2022)Yang, Orvieto, Lucchi, and He]{yang2022faster}
J.~Yang, A.~Orvieto, A.~Lucchi, and N.~He.
\newblock Faster single-loop algorithms for minimax optimization without strong concavity.
\newblock In \emph{International Conference on Artificial Intelligence and Statistics}, pages 5485--5517. PMLR, 2022.

\bibitem[Yurtsever et~al.(2015)Yurtsever, Tran~Dinh, and Cevher]{yurtsever2015universal}
A.~Yurtsever, Q.~Tran~Dinh, and V.~Cevher.
\newblock A universal primal-dual convex optimization framework.
\newblock In \emph{Advances in Neural Information Processing Systems}, volume~28, 2015.

\bibitem[Zhang et~al.(2020)Zhang, Xiao, Sun, and Luo]{zhang2020single}
J.~Zhang, P.~Xiao, R.~Sun, and Z.~Luo.
\newblock A single-loop smoothed gradient descent-ascent algorithm for nonconvex-concave min-max problems.
\newblock In \emph{Advances in Neural Information Processing Systems}, pages 7377--7389, 2020.

\bibitem[Zhang et~al.(2022)Zhang, Hong, and Zhang]{zhang2022lower}
J.~Zhang, M.~Hong, and S.~Zhang.
\newblock On lower iteration complexity bounds for the convex concave saddle point problems.
\newblock \emph{Mathematical Programming}, 194\penalty0 (1):\penalty0 901--935, 2022.

\bibitem[Zheng et~al.(2023)Zheng, Zhu, So, Blanchet, and Li]{zheng2023universal}
T.~Zheng, L.~Zhu, A.~M.-C. So, J.~Blanchet, and J.~Li.
\newblock Universal gradient descent ascent method for nonconvex-nonconcave minimax optimization.
\newblock In \emph{Advances in Neural Information Processing Systems}, pages 54075--54110, 2023.

\end{thebibliography}
\newpage
\section*{Appendix} 
\appendix

In the Appendix, we provide the missing proofs for nonconvex minimax problems. Our analysis builds on the framework developed in \citet{li2022nonsmooth,zheng2023universal}, and incorporates several key lemmas from \citet{zheng2023universal}. Below, we highlight only the components that are specific to our new algorithms. 

\section{Some Useful Lemmas}\label{APP:B}
\begin{lemma}\label{lem:6}
For any $t\ge0$, if $r_{\x},r_{\y}> L$, we have
 \vspace{1mm}
\begin{enumerate}[label=(\roman*)] 
\setlength{\itemsep}{1pt}
\item
 $  \|\x^{t+1}-\x(\y^{t},\z^{t},\v^{t})\|\leq    \sigma_6 \|\x^t-\x^{t+1}\|+\frac{r_{\x}+L}{r_{\x}-L}\|\x^t-\x^{t-1}\| + \frac{L}{r_{\x}-L} \|\y^t-\y^{t-1}\|\\ +\frac{r_{\x}}{r_{\x}-L}\|\z^t-\z^{t-1}\|$,  
 \item $\|\y^{t+1}-\y_+^{t}(\z^{t},\v^{t})\|\leq \eta_{\y}  L (\sigma_6+1)\|\x^{t+1}-\x^{t}\|+  \frac{2\eta_{\y} L r_{\x} }{r_{\x}-L}\|\x^t-\x^{t-1}\| \\ + \eta_{\y}\left(\frac{  Lr_{\x}}{r_{\x}-L}+r_{\y}\right) \|\y^t-\y^{t-1}\|   +\frac{\eta_{\y} Lr_{\x}}{r_{\x}-L}\|\z^t-\z^{t-1}\| + \eta_{\y}  r_{\y}\|\v^{t-1}-\v^{t}\| $,
    \end{enumerate}
    \vspace{1mm}
    where $\sigma_6=\tfrac{1+2\eta_{\x} r_{\x}}{\eta_{\x} (r_{\x}- L)}$ and $\y_{+}(\z,\v)=\proj_{\Y}(\y+\eta_{\y} \nabla_{\y}F(\x(\y,\z,\v),\y,\z,\v))$.
\end{lemma}

\begin{proof} 
We first prove part (i). From the update rule \[\x^{t+1}=\proj_{\X} \left(\x^t-2\eta_{\x}\nabla_{\x} F(\x^t,\y^t,\z^t,\v^t)  +\eta_{\x}\nabla_{\x} F(\x^{t-1},\y^{t-1},\z^{t-1},\v^{t-1}) \right),\] the optimality condition of the projection onto the convex set $\mathcal{X}$ implies  
\begin{align}\label{eq:eb_x}
    \langle\x^t-2\eta_{\x} \nabla_{\x}F(\x^t,\y^t,\z^t,\v^t)+\eta_{\x} \nabla_{\x}F(\x^{t-1},\y^{t-1},\z^{t-1},\v^{t-1})-\x^{t+1},  
    \x^{t+1}-\x(\y^{t},\z^{t},\v^{t})\rangle \geq 0.
\end{align}
Moreover, since $F(\cdot, \y,\z,\v)$ is  convex and $\x(\y^{t},\z^{t},\v^{t})=\argmin_{\x\in \X} F(\x,\y^t,\z^t,\v^t)$, we have 
\begin{align}
    \label{eq:eb_x_1}
    \eta_{\x} \langle \nabla_{\x} F(\x(\y^{t},\z^{t},\v^{t}),\y^{t},\z^{t},\v^{t}), \x^{t+1}-\x(\y^{t},\z^{t},\v^{t})\rangle \geq 0.
\end{align}
By adding \eqref{eq:eb_x} and \eqref{eq:eb_x_1}, we obtain 
\begin{equation}
\begin{aligned}
\langle &\x^t-\x^{t+1}+\eta_{\x} \nabla_{\x} F(\x^{t+1},\y^t,\z^t,\v^t)-2\eta_{\x} F(\x^t,\y^t,\z^t,\v^t)  +\eta_{\x} F(\x^{t-1},\y^{t-1},\z^{t-1},\v^{t-1}),\notag \\
    &\x^{t+1}-\x(\y^{t},\z^{t},\v^{t})\rangle \notag\\
    \geq\ & \eta_{\x} \langle \nabla_{\x} F(\x^{t+1},\y^t,\z^t,\v^t)- \nabla_{\x} F(\x(\y^{t},\z^{t},\v^{t}),\y^{t},\z^{t},\v^{t}), \x^{t+1}-\x(\y^{t},\z^{t},\v^{t})\rangle.
\end{aligned}
\end{equation}
By the $\left(r_{\x}-L\right)$-strong convexity of $F(\cdot,\y^t,\z^t,\v^t)$ and the Cauchy–Schwarz inequality, we obtain 
\begin{align*}
     &\eta_{\x} \left(r_{\x}-L\right) \|\x^{t+1}-\x(\y^{t},\z^{t},\v^{t})\|\\
     \leq\ & \|\x^t-\x^{t+1}\|   +\eta_{\x} \|\nabla_{\x} F(\x^{t+1},\y^t,\z^t,\v^t)- \nabla_{\x} F(\x^t,\y^t,\z^t,\v^t)\|\\
     & +\eta_{\x} \|\nabla_{\x} F(\x^t,\y^t,\z^t,\v^t)-\nabla_{\x}F(\x^{t-1},\y^{t-1},\z^{t-1},\v^{t-1}) \|\\
     \leq\ & \left(1+\eta_{\x} L+\eta_{\x} r_{\x}\right) \|\x^t-\x^{t+1}\|+\eta_{\x}\left(r_{\x}+L\right) \|\x^t-\x^{t-1}\|+\eta_{\x} L \|\y^t-\y^{t-1}\| +\eta_{\x} r_{\x}\|\z^t-\z^{t-1}\|, 
\end{align*}
where the second inequality follows from Assumption \ref{ass:1}.
Therefore, we have 
\begin{align*}
    \|\x^{t+1}-\x(\y^{t},\z^{t},\v^{t})\|\leq\ & \frac{1+\eta_{\x} L+\eta_{\x} r_{\x}}{\eta_{\x} \left(r_{\x}-L\right)} \|\x^t-\x^{t+1}\|+\frac{r_{\x}+L}{r_{\x}-L}\|\x^t-\x^{t-1}\|\\
    &+ \frac{L}{r_{\x}-L} \|\y^t-\y^{t-1}\|+\frac{r_{\x}}{r_{\x}-L}\|\z^t-\z^{t-1}\|.
\end{align*}
  Now consider part (ii).  By the definition of $\y_+^{t}(\z^{t},\v^{t})$ and the update rule of $\y^{t+1}$, we have
\begin{align}
    &\|\y^{t+1}-\y_+^{t}(\z^{t},\v^{t})\| \notag\\
    =\ & \|\proj_{\Y}(\y^t+2\eta_{\y} \nabla_{\y}F(\x^{t},\y^t,\z^t,\v^t)-\eta_{\y} \nabla_{\y}F(\x^{t-1},\y^{t-1},\z^{t-1},\v^{t-1}) ) \notag\\
    &-\proj_{\Y}(\y^t+\eta_{\y} \nabla_{\y}F(\x(\y^t,\z^t,\v^t),\y^t,\z^t,\v^t))\| \notag\\
    \leq\ & \eta_{\y}\|\nabla_{\y}F(\x^{t},\y^t,\z^t,\v^t)-\nabla_{\y}F(\x(\y^t,\z^t,\v^t),\y^t,\z^t,\v^t)\| \notag\\
    &+\eta_{\y}  \| \nabla_{\y} F(\x^{t},\y^{t},\z^{t},\v^{t})-\nabla_{\y} F(\x^{t-1},\y^{t-1},\z^{t-1},\v^{t-1})\|\notag\\ 
     \leq\ & \eta_{\y}  L  (\sigma_6+1)\|\x^{t+1}-\x^{t}\|+  \frac{2\eta_{\y} L  r_{\x} }{r_{\x}-L}\|\x^t-\x^{t-1}\|   + \eta_{\y}\left(\frac{  L r_{\x}}{r_{\x}-L }+r_{\y}\right) \|\y^t-\y^{t-1}\| \notag \\
     & +\frac{\eta_{\y} L r_{\x}}{r_{\x}-L }\|\z^t-\z^{t-1}\| + \eta_{\y}  r_{\y}\|\v^{t-1}-\v^{t}\| ,
\end{align}
where the first inequality follows from the nonexpansiveness of the projection operator, and the last inequality follows from Assumption~\ref{ass:1} and part~(i). 
\end{proof} 
By following the same procedure as in \citet[Lemmas 5 and 6]{zheng2023universal}, we obtain Lemma~\ref{lem:dec_1} and Lemma~\ref{lem:dec_2}, tailored to DS-OGDA. For simplicity, we omit the detailed proofs.
\begin{lemma}[Primal Descent]\label{lem:dec_1} For any $t\geq 0$, the following inequality holds:
    \begin{align*}
       & F(\x^{t},\y^{t},\z^{t},\v^{t})-F(\x^{t+1},\y^{t+1},\z^{t+1},\v^{t+1}) \notag\\
        \geq\   & \left(\frac{1}{\eta_{\x}}-\frac{2L+3r_{\x}}{2} \right)\|\x^t-\x^{t+1}\|^2  +\frac{r_{\y}-L}{2}\|\y^t-\y^{t+1}\|^2  +\frac{(1- \beta_{\x})r_{\x}}{ \beta_{\x}}\|\z^t-\z^{t+1}\|^2 \notag \\
        & + \frac{(\beta_{\y}-2)r_{\y}}{2\beta_{\y}}\|\v^t-\v^{t+1}\|^2    -\frac{3  (L+r_{\x})}{2 }\|\x^{t-1}-\x^t\|^2 -\frac{3 (L+r_{\x})}{2 }\|\y^{t-1}-\y^t\|^2 \notag \\
        &   -\frac{3 (L+r_{\x})}{2 }\|\z^{t-1}-\z^t\|^2 +\langle \nabla_{\y} F(\x^{t+1},\y^t,\z^{t+1},\v^{t+1}), \y^t- \y^{t+1}\rangle.
    \end{align*}
\end{lemma}
 \begin{lemma}[Dual Ascent]\label{lem:dec_2} For any $t\geq 0$, the following inequality holds:
   \[
   \begin{aligned}
       & d(\y^{t+1},\z^{t+1},\v^{t+1})- d(\y^t,\z^t,\v^t)\\
       \geq\ &  \frac{(2-\beta_{\y})r_{\y}}{2\beta_{\y}}\|\v^{t+1}-\v^t\|^2+ \frac{r_{\x}}{2}\langle \z^{t+1}+\z^t-2\x(\y^{t+1},\z^{t+1},\v^{t+1}),  \z^{t+1}-\z^t \rangle\\
       &+ \langle \nabla_{\y} F(\x(\y^t,\z^t,\v^{t+1}),\y^t,\z^t,\v^{t+1}),\y^{t+1}-\y^t \rangle-\frac{L_d}{2}\|\y^{t+1}-\y^t\|^2,
   \end{aligned}
   \]
   where $L_d\coloneqq L \sigma_1+L +r_{\y}$ and $\sigma_1=\tfrac{r_{\x}}{r_{\x}-L}$.
\end{lemma}
\section{Proof of Proposition \ref{proposition:suff_dec}}\label{APP:D}
\begin{proof}
The overall argument parallels that of \citet[Theorem 1]{zheng2023universal}; the only changes required for DS-OGDA appear in Lemmas~\ref{lem:6}, \ref{lem:dec_1}, and~\ref{lem:dec_2}. The detailed derivation is omitted, and we directly obtain:
\begin{align}\label{eq:suff_mid}
   \Phi^t-\Phi^{t+1}  
    \geq\ &\left(s_1^{\x}-5\eta_{\y} ^2 L^2 (\sigma_6+1)^2 s_1^{\y}\right)  \|\x^t-\x^{t+1}\|^2+ \left(s_2^{\x}- \frac{20\eta_{\y}^2 L^2 r_{\x}^2 }{\left(r_{\x}-L\right)^2} s_1^{\y}\right) \|\x^t-\x^{t-1}\|^2\notag \\
    &+\left(\frac{s_1^{\y}}{2}- 2\beta_{\y}^2\sigma_8^2s_1^{\v}-6r_{\x}\kappa\sigma_1^2\sigma_8^2\right)\|\y^t-\y_+^t(\z^t,\v^t)\|^2 \notag\\
    & + \left(s_2^{\y}- 5\eta_{\y}^2\left(\frac{  Lr_{\x}}{r_{\x}-L}+r_{\y}\right)^2 s_1^{\y}\right) \|\y^t-\y^{t-1}\|^2\notag\\
     &+\left(s_1^{\z}-2\beta_{\y}^2\sigma_3^2 s_1^{\v}\right) \|\z^t-\z^{t+1}\|^2+ \left(s_2^{\z}- \frac{5\eta_{\y}^2 L^2 r_{\x}^2 }{\left(r_{\x}-L\right)^2}s_1^{\y}\right) \|\z^t-\z^{t-1}\|^2   \notag\\
    &  +  \left(\frac{s_1^{\v}}{2}-12r_{\x}\kappa \sigma_1^2\sigma_5^2\right)\|\v^t-\v_+^t(\z^{t+1})\|^2 + \left(s_2^{\v}-5\eta_{\y} ^2r_{\y}^2s_1^{\y}\right)  \|\v^t-\v^{t-1}\|^2 \notag\\
    &- 2r_{\x}\kappa
          \|\x(\z^{t+1},\v(\z^{t+1}))-\x(\z^{t+1},\v_+^t(\z^{t+1}))\|^2, 
\end{align}
where $\kappa>0$ is any positive constant, and the coefficients are defined as follows:
\begin{alignat*}{2}
    &\ \ \sigma_1=\sigma_2  = \frac{r_x}{r_x - L}, 
    &\quad
    &s_1^{\x}  = \frac{1}{2\eta_{\x}}-\frac{6L +3r_{\x}+2r_{\y}}{2}, \\
   &\ \ \sigma_3  = \frac{r_{\x}\sigma_1}{r_{\y}-L }+1,
    &\quad
   &s_1^{\z}  = \frac{\left(1-\beta_{\x}(2+4\sigma_2)\right)r_{\x}}{2\beta_{\x}}-\frac{r_{\x}}{\kappa}-12r_{\x}\kappa \sigma_1^2\sigma_3^2,  \\
   &\ \ \sigma_5  = \frac{r_{\y}}{r_{\y}-L},
    &\quad
    &s_1^{\v}  = \frac{\left(1-\beta_{\y}\right)r_{\y}}{2\beta_{\y}}-L-r_{\y}, \\
  & \ \ L_d  = L \sigma_1+L +r_{\y},
    &\quad
    &s_2^{\x}  = \frac{1}{2\eta_{\x}}-\frac{3(2L+r_{\x}+r_{\y})}{2}-\frac{L\left(r_{\x}+L\right)}{r_{\x}-L}, \\
    &\ \ \sigma_8  = \frac{1+\eta_{\y} L_d}{\eta_{\y}  (r_{\y}-L)},
    &\quad
    &s_2^{\z} =\frac{r_{\x}}{2\beta_{\x}}-\frac{3(L+r_{\x})}{2}-\frac{L r_{\x} }{r_{\x}-L},  
\end{alignat*}
\begin{align*}
    &s_2^{\y} = \frac{1}{2\eta_{\y}}-\frac{3(2L+r_{\x}+r_{\y})}{2}-\frac{L^2}{r_{\x}-L},  \quad s_2^{\v} =\frac{r_{\y}}{2\beta_{\y}}-\frac{3(L+r_{\y})}{2},\\
    &s_1^{\y} =\frac{1 }{2\eta_{\y} }-L_d-L\sigma_6^2-\frac{r_{\y}+3L}{2}-6r_{\x} \kappa \sigma_1^2-\frac{2L\left(r_{\x}+L\right)}{r_{\x}-L}.
\end{align*}
We now turn to a detailed step size analysis. Under the assumption $r_{\x}\geq 2L$ and $r_{\y}\geq 2L$, the step size choices specified in Condition \ref{cond:step} can be utilized to ensure all coefficients remain  bounded from below:
   \begin{align*}
       &  \eta_{\x}\leq  \frac{1}{6\left(4L+r_{\x} +r_{\y}\right)} \Rightarrow \   s_2^{\x}\geq \frac{1}{4\eta_{\x}}, \\
       &\eta_{\y} \leq \frac{1}{ 6\left(3L+r_{\x} +r_{\y}\right)} \Rightarrow \  
        s_2^{\y} \geq \frac{1}{4\eta_{\y}}, \\
        &\beta_{\x} \leq \frac{r_{\x}}{8\left(L+r_{\x}+L\right)} \Rightarrow \  
        s_2^{\z}\geq  \frac{r_{\x}}{4\beta_{\x}}, \\
       & \beta_{\y}\leq \frac{r_{\y}}{6\left(L+r_{\y}\right)} \Rightarrow \  s_2^{\v}\geq \frac{r_{\y}}{4\beta_{\y}} \, \text{and } \, 
         s_1^{\v} \geq \frac{r_{\y}}{4\beta_{\y}},
   \end{align*}
   where the first implication follows from $\tfrac{ r_{\x}+L }{r_{\x}-L}\leq 3$ and the second implication is derived by $\tfrac{L }{r_{\x}-L}\leq 1$. 
Moreover, due to $r_{\x}\geq 2L$ and $r_{\y}\geq 2L$, we have 
\[\sigma_1 = \sigma_2 = \frac{L}{r_{\x} - L} + 1 \leq 2,  \sigma_3 = \frac{r_{\x} \sigma_1}{r_{\y} - L}+1\leq \frac{2r_{\x}}{L} + 1, \text{and } \sigma_5 = \frac{r_y}{r_y - L} \leq 2. \] 
Then, we set $\kappa=4\beta_{\x}$. Using again the step size conditions in Condition~\ref{cond:step}, specifically, $\beta_{\x}\leq \tfrac{\eta_{\y}  L^2}{7680r_{\x} } $ and $\beta_{\y}\leq \tfrac{r_{\y}}{2L^2}$, we obtain the coefficient before the term $\|\v^t-\v_+^t(\z^{t+1})\|^2$,  
    \begin{align*}
         \frac{s_1^{\v}}{2}-12r_{\x}\kappa\sigma_1^2\sigma_5^2\geq\ & \frac{r_{\y}}{8\beta_{\y}}-768r_{\x}\beta_{\x} \geq \frac{r_{\y}}{8\beta_{\y}}-\frac{\eta_{\y} L^2}{10}   \geq \frac{r_{\y}}{16\beta_{\y}}. 
    \end{align*}    
    We next bound the coefficient preceding the term  $\|\z^t-\z^{t+1}\|^2$, i.e., 
   \begin{align*}
        s_1^{\z}-2\beta_{\y}^2\sigma_3^2 s_1^{\v}   =\ & \frac{\left(1-\beta_{\x}(2+4\sigma_2)\right)r_{\x}}{2\beta_{\x}}-\frac{r_{\x}}{4\beta_{\x}}-48r_{\x}\beta_{\x}\sigma_1^2\sigma_3^2-2\beta_{\y}^2\sigma_3^2 s_1^{\v} \notag\\
       \geq\ &\frac{r_{\x}}{2\beta_{\x}}\left(\frac{1}{2}-10\beta_{\x} \right) -48r_{\x}\beta_{\x}\sigma_1^2\sigma_3^2- r_{\y}\beta_{\y} \sigma_3^2 \notag\\
       \geq\ &\frac{r_{\x}}{2\beta_{\x}}\left(\frac{1}{2}-10\beta_{\x} \right) -\left(192r_{\x}\beta_{\x} +r_{\y}\beta_{\y}\right)  \left(\frac{4r_{\x}^2 }{L^2}+1+\frac{4r_{\x} }{L}\right) \notag\\
\geq\ &\frac{r_{\x}}{2\beta_{\x}}\left(\frac{1}{2}-10\beta_{\x} \right)-\frac{3\eta_{\y} L^2}{80}\left(\frac{4r_{\x}^2 }{L^2}+1+\frac{4r_{\x} }{L}\right)  \geq \frac{r_{\x}}{8\beta_{\x}},        
    \end{align*}
 where the first inequality follows from $\sigma_2\leq 2$ and $\beta_{\y}^2s_1^{\v}\leq \tfrac{r_{\y}\beta_{\y}}{2}$, the second one is due to the definition of $\sigma_3$ and $\sigma_1\leq 2$, the third one follows from the step size conditions in Condition \ref{cond:step}, i.e., $\beta_{\x}\leq \tfrac{\eta_{\y}  L^2}{7680r_{\x} } $  and $\beta_{\y}\leq \tfrac{\eta_{\y}  L^2}{240r_{\y} }$, and the last one is a consequence of the condition $\beta_{\x} \leq \tfrac{2r_{\x}}{80r_{\x} +6r_{\x}^2 +3L^2+12r_{\x}L}$  in Condition~\ref{cond:step}.

We proceed to bound the coefficient of $\|\y^t-\y_+^t(\z^t,\v^t)\|^2$ by first estimating  $s_1^{\y}$ as follows: 
\begin{align}\label{eq:s1y}
    s_1^{\y}=\ &\frac{1 }{2\eta_{\y} }-L_d-L\sigma_6^2-\frac{r_{\y}+3L}{2}-\frac{2L\left(r_{\x}+L\right)}{r_{\x}-L}-24r_{\x} \beta_{\x} \sigma_1^2 \notag\\
    \geq \ & \frac{1}{6\eta_{\y} } -96r_{\x}\beta_{\x} 
    \geq\  \frac{1}{6\eta_{\y}}-\frac{\eta_{\y} L^2}{80}\geq \frac{1}{12\eta_{\y}}.
\end{align}
Here, the first inequality is derived by the step size conditions in Condition~\ref{cond:step}, i.e., $\tfrac{1}{\eta_{\y} }\geq \max\{6L(\sigma_6+1)^2 , 3(2L_d+r_{\y}+15L)\}$, which implies $\tfrac{1}{2\eta_{\y} }-L(\sigma_6+1)^2\geq \tfrac{1}{3\eta_{\y} }$ and $\tfrac{1}{3\eta_{\y} }-L_d-\tfrac{r_{\y}+3L}{2}-\tfrac{2L\left(r_{\x}+L\right)}{r_{\x}-L}\geq\tfrac{1}{3\eta_{\y} }-L_d-\tfrac{r_{\y}+15L}{2} \geq  \tfrac{1}{6\eta_{\y} }$; The second inequality follows from the step size conditions in Condition~\ref{cond:step}, i.e., $\beta_{\x}\leq \tfrac{\eta_{\y}  L^2}{7680r_{\x} } $; The last inequality follows from the condition $\eta_{\y} \leq \sqrt{\tfrac{20}{3}}\tfrac{1}{L}$ implied by Condition~\ref{cond:step}. Secondly, we bound $\sigma_8$ as below: 
\begin{align}
    \label{eq:sigma_8}
    \sigma_8=\frac{\frac{1}{\eta_{\y} }+L_d}{r_{\y}-L}\leq \frac{2}{\eta_{\y}  r_{\y}}+\frac{L \sigma_1+L +r_{\y}}{r_{\y}-L}\leq \frac{2}{\eta_{\y}  r_{\y}}+\frac{2(3L+r_{\y})}{r_{\y}}=\frac{2}{\eta_{\y}  r_{\y}}+5.
\end{align}
We are now in position to bound the complete coefficient of $\|\y^t-\y_+^t(\z^t,\v^t)\|^2$:  
\begin{align*}
     \frac{s_1^{\y}}{2}- 2\beta_{\y}^2\sigma_8^2s_1^{\v}-24r_{\x}\beta_{\x}\sigma_1^2\sigma_8^2 
    \geq\ &\frac{1}{24\eta_{\y} } -\left(r_{\y}\beta_{\y}+96r_{\x}\beta_{\x}\right)  \left(\frac{4}{\eta_{\y} ^2 r_{\y}^2}+25+\frac{20}{\eta_{\y}  r_{\y}}\right)\\
    \geq\ & \frac{1}{24\eta_{\y} }-\frac{\eta_{\y}  L^2}{60}\left(\frac{4}{\eta_{\y} ^2 r_{\y}^2}+25+\frac{20}{\eta_{\y}  r_{\y}}\right)\geq \frac{1}{60\eta_{\y}},
\end{align*}
where the first inequality follows from \eqref{eq:s1y}, \eqref{eq:sigma_8}, and the bound $\beta_{\y}^2s_1^{\v}\leq \tfrac{r_{\y}\beta_{\y}}{2}$; The second inequality is derived by the step size conditions in Condition~\ref{cond:step}, i.e., $\beta_{\x}\leq \tfrac{\eta_{\y}  L^2}{7680r_{\x} } $  and $\beta_{\y}\leq \tfrac{\eta_{\y}  L^2}{240r_{\y} }$; The last inequality uses the assumption $r_{\y}\geq 2L_{\y}$ and  the step size condition $\eta_{\y}\leq \tfrac{1}{20(5L^2 +2L)}$ in Condition~\ref{cond:step}. 

Finally, we bound the remaining coefficients appearing in other terms. These bounds  follow directly from the step size conditions in Condition~\ref{cond:step} and the earlier estimate $s_1^{\y}\leq \tfrac{1}{2\eta_{\y}}$. It is then straightforward to verify that:
\begin{align*}
     &\eta_{\x}\leq \frac{1}{2(7L +3r_{\x} +2r_{\y})} \Rightarrow \s_1^{\x}-5\eta_{\y} ^2 L^2 (\sigma_6+1)^2 s_1^{\y} \geq\frac{1}{4\eta_{\x}}\\
    &\eta_{\y} \leq  \frac{1}{ 4\sqrt{5}r_{\y}} \Rightarrow\   s_2^{\y}- 5\eta_{\y}^2\left(\frac{  Lr_{\x}}{r_{\x}-L}+r_{\y}\right)^2 s_1^{\y} \geq \frac{1}{8\eta_{\y}},\\
    & \beta_{\y}\leq \frac{1}{32r_{\y}\eta_{\y}} \Rightarrow \  s_2^{\v}-5\eta_{\y} ^2 r_{\y}^2 s_1^{\y} \geq \frac{r_{\y}}{8\beta_{\y}},\\
    & \beta_{\x}\leq \frac{r_{\x}}{24r_{\y}^2\eta_{\y}} \Rightarrow \   s_2^{\z}- \frac{5\eta_{\y}^2 L^2 r_{\x}^2 }{\left(r_{\x}-L\right)^2}s_1^{\y}\geq \frac{r_{\x}}{8\beta_{\x}},\\
    & \eta_{\x}\leq \frac{1}{320r_{\y}^2\eta_{\y}}   \Rightarrow \    s_2^{\x}-\frac{20\eta_{\y}^2 L^2 r_{\x}^2 }{\left(r_{\x}-L\right)^2} s_1^{\y}\geq  \frac{1}{8\eta_{\x}},
\end{align*}
where the first implication uses the condition $\tfrac{1}{\eta_{\y} }\geq  6L(\sigma_6+1)^2$ in Condition~\ref{cond:step}.

Putting everything together yields \eqref{eq:suff_dec}. We finished the proof. 
\end{proof}

\end{document}